\documentclass[11pt]{amsart}

\usepackage[a4paper,inner=3.6cm,outer=3.6cm,top=3.7cm,bottom=3.7cm]{geometry}
\usepackage{amsfonts}
\usepackage{amssymb}
\usepackage{amsmath}
\usepackage[all,cmtip,2cell]{xy}
\usepackage{comment}
\usepackage{array}
\usepackage{enumitem}
\usepackage{mathtools}
\usepackage{multirow}
\usepackage{pinlabel}
\usepackage{marginnote}
\setlist[itemize]{leftmargin=*}
\setlist[enumerate]{leftmargin=*}

\usepackage{color}

\usepackage{hyperref}  
\hypersetup{%
  colorlinks=true,%
  linkcolor=blue,%
  citecolor=blue,%
  filecolor=blue,
}

\usepackage{xfrac}
\usepackage{faktor} 

\def\co{\colon\thinspace}
\newcommand{\Q}{\mathbb{Q}}
\newcommand{\C}{\mathbb{C}}
\newcommand{\Z}{\mathbb{Z}}
\newcommand{\N}{\mathbb{N}}

\newcommand{\R}{\mathbb{R}}

\renewcommand{\l}{\vartheta}

\newcommand{\bd}{\partial}

\newcommand{\M}{\mathcal{M}}

\newcommand{\Symp}{\text{\it Symp}}

\newcolumntype{C}[1]{>{\centering\arraybackslash$}p{#1}<{$}}

\renewcommand{\k}{\mathbb{K}}
\newcommand{\x}{\times}
\newcommand{\CO}{\mathcal{{C}{O}}}
\newcommand{\OC}{\mathcal{{O}{C}}}
\newcommand{\CC}{\mathcal{C}}

\newcommand{\ev}{\mathrm{ev}}

\newcommand{\RP}{\R P}
\newcommand{\CP}{\C P}

\newcommand{\Cl}{\mathit{Cl}}

\newcommand{\clow}{_{\mathit{low}}^{(2)}}
\newcommand{\low}{{\mathit{low}}}

\newcommand{\pt}{\mathrm{pt}}
\newcommand{\PxP}{\CP^1 \times \CP^1}
\newcommand{\ad}{\mathrm{ad}}
\newcommand{\Sa}{S^2_{\ad}}

\newcommand{\PO}{\mathfrak{PO}}
\newcommand{\bb}{\mathfrak{b}}
\newcommand{\T}{\hat{T}}

\newcommand{\hgamma}{\hat{\gamma}}
\newcommand{\homega}{\hat{\omega}}
\newcommand{\cH}{H}

\newcommand{\hF}{\hat F}
\newcommand{\hG}{\hat G}

\renewcommand{\c}{^{(2)}}
\newcommand{\J}{\mathcal{J}}

\newcommand{\BlI}{\mathbb{C}P^2\# \overline{\mathbb{C}P^2}}

\newcommand{\BlIII}{\mathbb{C}P^2\# 3\overline{\mathbb{C}P^2}}

\newcommand{\del}{\partial}

\newtheorem{theorem}{Theorem}[section]
\newtheorem*{theorem*}{Theorem}
\newtheorem{proposition}[theorem]{Proposition}
\newtheorem{lemma}[theorem]{Lemma}

\newtheorem{conjecture}[theorem]{Conjecture}

\theoremstyle{definition}
\newtheorem{definition}[theorem]{Definition}

\theoremstyle{remark}
\newtheorem{remark}{Remark}[section]

\numberwithin{equation}{section}

\begin{document}

\title[Low-area Floer theory]{Low-area Floer theory and non-displaceability}
\author{Dmitry Tonkonog}
\email{dt385@cam.ac.uk / dtonkonog@gmail.com }
\address{Department of Pure Mathematics and Mathematical Statistics,
University of Cambridge,
Wilberforce Road,
Cambridge
CB3 0WB, UK}
\author{Renato Vianna}
\email{r.vianna@dpmms.cam.ac.uk}
\address{Department of Pure Mathematics and Mathematical Statistics,
University of Cambridge,
Wilberforce Road,
Cambridge
CB3 0WB, UK}

\begin{abstract} We introduce a new version of Floer theory of a non-monotone
	Lagrangian submanifold which only uses least area holomorphic disks with boundary on
	it. We use this theory to prove non-displaceability theorems about continuous families of Lagrangian tori in the complex projective plane and other del Pezzo surfaces.
\end{abstract}

\maketitle



\section{Introduction}

In this paper, we introduce a new non-displaceability criterion for Lagrangian
submanifolds of a given compact symplectic (four-)manifold. The criterion is
based on a Hamiltonian isotopy invariant of a Lagrangian
submanifold constructed using $J$-holomorphic disks \emph{of lowest area}
with boundary on the Lagrangian. We apply our criterion to prove
non-displaceability results for Lagrangian tori in $\CP^2$ and in other del
Pezzo surfaces, for which there is no clear alternative proof using
standard results in Floer theory.
\label{sec:intro}
We shall focus on dimension~four in line with our intended applications, although we do have a clear vision of a higher-dimensional setup to which our methods generalise---it is mentioned briefly in Subsection~\ref{subsec:higher_dim}.

\subsection{Challenges in Lagrangian rigidity} A classical question in symplectic
topology, originating from Arnold's conjectures and still inspiring numerous
advances in the field, is to understand whether two given Lagrangian
submanifolds $L_1$, $L_2$ are (Hamiltonian) non-displaceable, meaning that there
exists no Hamiltonian diffeomorphism that would map $L_1$ to a Lagrangian
disjoint from $L_2$. It is sometimes referred to as the {\it Lagrangian rigidity} 
problem, and the main tool to approach it is Floer theory. 
Historically, most applications of Floer theory were focused on monotone
(or exact) Lagrangians, as in those cases it is foundationally easier to set up,
and usually easier to compute. 

More recent developments have given access to non-displaceability results
concerning non-monotone Lagrangians. One of such developments is called Floer
cohomology with bulk deformations, introduced by Fukaya, Oh, Ohta and Ono
\cite{FO3Book,FO311a}. Using bulk deformations, the same authors \cite{FO312}
found a {\it continuous} family of non-displaceable
Lagrangian tori $\T_a$ in $\CP^1\times \CP^1$, indexed by $a\in(0,1/2]$.
(When we say that a Lagrangian is
non-displaceable, we mean that it is non-displaceable from itself.) For some
other recent methods, see \cite{AM13, Bor13,Wo11}.

\begin{remark}
	
 To be able to observe such ``rigidity for families'' phenomena, it is essential
 to consider non-monotone Lagrangian submanifolds, as spaces of monotone ones up
 to Hamiltonian isotopy are discrete, on compact symplectic manifolds.
 \end{remark}

It is easy to produce challenging instances of the displaceability problem which
known tools fail to answer. For example, consider the 2:1
cover $\CP^1\times \CP^1\to \CP^2$ branched along a conic curve. Taking the
images of the above mentioned tori under this cover, we get a
family of Lagrangian tori in the complex projective plane denoted by $T_a\subset
\CP^2$ and indexed by $a\in(0,1/2]$ --- see Section \ref{sec:T_a} and \cite[Section~3]{Wu15} for details.
It turns out that the tori $T_a \subset \CP^2$ have
trivial bulk-deformed Floer cohomology for any bulk class $\bb \in H^2(\CP^2,
\Lambda_0)$, as we check in Proposition~\ref{prop: Bulk CP^2}. While one can
show that the tori $T_a$ are displaceable when $a>1/3$, the following remains to
be a conjecture.

\begin{conjecture} 
	\label{con:T_a}	
	For each $a\in (0,1/3]$, the Lagrangian torus $T_a\subset
	\CP^2$ is Hamiltonian non-displaceable. 
\end{conjecture}	

Motivated by this and similar problems, we introduce a new approach, called
low-area Floer theory, to solve rigidity problems concerning some non-monotone
Lagrangians. 
Let us list some application of this technique.

\begin{theorem}
	\label{th:T_a}
	For each $a\in (0,1/9]$, the torus $T_a\subset \CP^2$ is Hamiltonian
	non-displaceable from the monotone Clifford torus $T_\Cl\subset \CP^2$.
\end{theorem}

\begin{remark}
	An interesting detail of the proof, originating from Lemma~\ref{lem:cotangent_bundles}(ii), is that we use $\Z/8$ coefficients for our Floer-theoretic invariants,
	and it is impossible to use a field, or the group $\Z$, instead. To place this into context, recall
	that conventional Floer cohomology over finite fields can detect non-displaceable
	monotone Lagrangians unseen by characteristic zero fields: the simplest example
	is $\RP^n\subset \CP^n$, see e.g.~\cite{FO314}; a more sophisticated example,
	where the characteristic of the field to take is not so obvious, is the Chiang 
	Lagrangian studied by Evans and~Lekili \cite{EL15b}, see also J.~Smith~\cite{Smi15}. 
However, there are no examples in conventional Floer
	theory that would require working over a torsion group which is not a field.  

\end{remark}

We can also show analogous results for some other del Pezzo surfaces. They are summarised below, and the precise formulations are contained in Theorems~\ref{th: Ta in PxP},~\ref{th: Ta in BlIII}. 

\begin{theorem}
In $\C P^1\times \C P^1$ and in $\BlIII$ with a monotone symplectic form, there exists a one-parametric family of Lagrangian tori which are non-displaceable from the standard monotone toric fibre.
\end{theorem}

The next result exhibits a {\it two-parametric} family of non-displaceable
Lagrangian tori in symplectic del Pezzo surfaces. By a symplectic del~Pezzo
surfaces we mean \emph{monotone} symplectic 4-manifolds, whose classification follows
from a series of works \cite{MD90,MD96,Ta95,Ta96,Ta00Book,OhtaOno96,OhtaOno97}.
Recall that their list consists of blowups of $\CP^2$ at $0\le k\le 8$ points,
and of $\CP^1\times\CP^1$. By the correspondence between monotone symplectic 
4-manifolds and complex Fano surfaces we will omit the term ``symplectic'', and 
call them del Pezzo surfaces from now on.

\begin{theorem}
	\label{th:non_displ_spheres}
	Let $X$ be a del Pezzo surface and $S,S'\subset X$ be a pair of Lagrangian spheres with homological intersection $[S]\cdot[S']=1$. Then, for some $0<a_0,b_0<1/2$, there exist two families of Lagrangian tori indexed by $a,b$: 
	$$T_a,\ T'_b\subset X,\quad a\in(0,a_0),\ b\in(0,b_0),$$
	lying in a neighbourhood of the sphere $S$ resp.~$S'$, and such that $T_a$ is non-displaceable from $T_b'$ for all $a,b$ as above.  
\end{theorem}

In our construction, any two different tori in the same family $\{T_a\}$ will be disjoint, and the same will hold for the $\{T_b'\}$.

 Recall that pairs of once-intersecting Lagrangian spheres exist inside
 $k$-blowups of $\CP^2$ when $k\ge 3$. For example, one can take Lagrangian
 spheres with homology classes $[E_i]-[E_j]$ and $[E_j]-[E_k]$ where
 $\{E_i,E_j,E_k\}$ are three distinct exceptional divisors, as explained in \cite{SeiL4D,Ev10}.
 These spheres
 can also be seen from the almost toric perspective \cite{Vi16}: on an
 almost toric base diagram for the corresponding symplectic del Pezzo surface,
 these Lagrangian spheres projects onto the segment connecting two nodes on the same monodromy line. For homology reasons, blowups of $\CP^2$ at 2 or less points do not contain such pairs of spheres.

\subsection{Lagrangian rigidity from low-area Floer theory} \label{subsec: DefOC_low}
We call a symplectic manifold $X$ monotone if $\omega$ and $c_1(X)$ give positively proportional classes in $H^2(X;\R)$.
Similarly, a Lagrangian is called monotone if the symplectic area class and the Maslov class of $L$ are positively proportional in $H^2(X,L;\R)$. It is quite common to use a definition where the proportionality is only required over $\pi_2(X)$ or $\pi_2(X,L)$; we stick to the homological version for convenience.

Floer theory for monotone Lagrangians has abundant algebraic structure, a
particular example of which are the open-closed and closed-open {\it string
maps}. There is a non-displace\-ability criterion for a pair of monotone
Lagrangians formulated in terms of these string maps; it is due to Biran and
Cornea and will be recalled later. Our main finding can be summarised as
follows: it is possible define a low-area version of the string maps for {\it
non-monotone} Lagrangians, and prove a version of Biran-Cornea's theorem under
an additional assumption on the areas of the disks involved. This method can
prove non-displaceability in examples having no clear alternative proof by means
of conventional Floer theory for non-monotone Lagrangians. We shall focus on
dimension 4, and proceed to a precise statement of our theorem.

Fix a ring $Q$ of coefficients; it will be used for all (co)homologies when the
coefficients are omitted. (The coefficient ring does not have to include a
Novikov parameter in the way it is done in classical Floer theory for
non-monotone manifolds; rings like $\Z/k\Z$ are good enough for our purpose.)
Let $L,K\subset X$ be two orientable, not necessarily monotone, Lagrangian
surfaces in a compact symplectic four-manifold $X$. 

Denote
\begin{equation}
	\label{eq:defn_a_dim4}
	a=\min\{\omega(u)>0\ | \ u\in H_2(X,L;\Z),\ \mu(u)=2\},
\end{equation}
assuming this minimum exists. This is the least positive area of topological Maslov index~2 cycles with boundary on $L$. (For example, we currently do not allow the above set of areas to have infimum equal to $0$.)
Also, denote by $A$ the next-to-the-least such area:
\begin{equation} \label{eq: def A}
	A = \min \{ \omega(u) > a \ |\  u\in H_2(X,L;\Z),\ \mu(u)=2\}.
\end{equation}
We assume that the minimum exists, including the possibility $A = +\infty = \min \emptyset$; the latter is the case when $L$ is monotone.  

Fix a tame almost complex structure $J$ and a point $p_L\in L$. Let
$\{D_i^L\}_i\subset (X,L)$ be the images of all $J$-holomorphic Maslov index~2
disks of area $a$ such that $p_L\in \bd D_i^L$ and \emph{whose boundary is
non-zero in $H_1(L;\Z)$} (their number is finite, by Gromov compactness \cite{Gr85}). 
Assume that 
\begin{equation}
\label{eq:disk_low_area_cancel_bdies} \sum_i \bd [D_i^L]=0\in H_1(L)
\end{equation} 
and the disks are regular with respect to $J$. Recall that by convention, the above
equality needs to hold over the chosen ring $Q$. Then let 
$$ 
\OC\clow([p_L])\in H_2(X) 
$$ 
be any element whose image under the map $ H_2(X)\to H_2(X,L) $ equals
$\sum_i [D^L_i]$. We call this class the {\it low-area string invariant of $L$.} 
Observe that it is defined only up to the image
$$
H_2(L)\to H_2(X),
$$
i.e.~up to $[L]\in H_2(X)$,  compare Remark \ref{rem:OC_def_upto}. In the cases of interest, we will have $[L]=0$; but regardless of this, we prove in Lemma~\ref{lem:OC_invt} that 
$\OC\clow([p_L])$ is independent of the choices we made, up to $[L]$.
Finally, consider $K$ instead of $L$ and define the numbers $b$ and $B$ analogously
to $a$ and $A$, respectively. Let $p_K$ be a point on $K$.

\begin{theorem}
	\label{th:CO}
	Assume that Condition~(\ref{eq:disk_low_area_cancel_bdies}) holds for $L$ and $K$, 
	the minima $a$, $A$, $b$, $B$ exist, and $[L]=[K]=0$. Suppose that $a+b<\min(A,B)$
	and the homological intersection number below is non-zero over $Q$: $$\OC\clow([p_L])\cdot \OC\clow([p_K])\neq 0.$$ Then $L$ and $K$ are Hamiltonian non-displaceable from each other.
\end{theorem}

Above, the dot denotes the intersection pairing $H_2(X)\otimes H_2(X)\to Q$. 
We refer to Subsection~\ref{subsec:non_displ_dim4} for a comparison with Biran-Cornea's theorem in the monotone setup, 
and for a connection of $\OC\clow$ with the classical open-closed string map.
Note that the above intersection number 
is well-defined due to the hypothesis $[L]=[K]=0$.

\begin{remark} \label{rmk:L.K=0} The condition $[L]=[K]=0$ is in fact totally
non-restrictive, due to the following two ``lower index'' non-displaceability
criteria. First, if $[L] \cdot [K] \neq 0$, then $L$ and $K$ are
non-displaceable for topological reasons. Second, if $$[L]\cdot
\OC\clow([p_K])\neq 0\quad \text{ or }\quad [K]\cdot \OC\clow([p_L])\neq 0,$$
one can show that $K$ and $L$ are non-displaceable by a variation on
Theorem~\ref{th:CO} whose proof can be carried analogously. Finally, if the
conditions of the previous two criteria fail, the intersection number
$\OC\clow([p_L])\cdot \OC\clow([p_K])$ is well-defined, and the reader can check
that the proof of Theorem~\ref{th:CO} still applies.

\end{remark}

Our proof of Theorem~\ref{th:CO} uses the idea of gluing holomorphic disks into
annuli and running a cobordism argument by changing the conformal parameter of
these annuli. This argument has been prevously used in Abouzaid's
split-generation criterion \cite{Ab10} and in Biran-Cornea's theorem
\cite[Section 8.2.1]{BC09B}. We follow the latter outline with several important
modifications involved. The condition $a+b<\min(A,B)$, which does not arise when
both Lagrangians are monotone ($A=B= +\infty$), is used in the proof when the
disks $D_i^L$ and $D_j^K$ are glued to an annulus of area $a+b$; the condition
makes sure higher-area Maslov index~2 disks on $L$ cannot bubble off this
annulus. This condition, for example, translates to $a<1/9$ in
Theorem~\ref{th:T_a}. Additionally, we explain how to achieve transversality for
the annuli by domain-dependent perturbations -- although arguments of
this type appeared before in the context of Floer theory \cite{SeiBook08,Ab10,She11,CM07,CW13},
we decided to explain this carefully.

\begin{remark}
	Our proof only uses
	classical transversality theory for holomorphic curves, as opposed to virtual
	perturbations required to set up conventional Floer theory for non-monotone
	Lagrangians; compare \cite{Cha15A}.
\end{remark}

\begin{remark}
There is another (weaker) widely used notion of monotonicity of $X$ or $L$,
where $\omega$ is required to be proportional to $c_1(X)$ resp.~Maslov class of
$L$ only when restricted to the image of $\pi_2(X)$ resp.~$\pi_2(X,L)$ under
the Hurewicz map. It is possible to use this definition throughout the paper; moreover, the numbers $a$ and $A$, see \eqref{eq:defn_a_dim4},
\eqref{eq: def A}, can be defined considering $u \in \pi_2(X,L)$. All theorems still hold as stated, except for small
adaptations in the statement of Lemma \ref{lem:cotangent_bundles}, e.g.
requiring $c_1(X)_{|\pi_2(X)} = k \omega_{|\pi_2(X)}$. 
\end{remark}

Next, we shall need a  technical improvement of our theorem. Fix a field $\k$, and choose an affine subspace
$$S_L\subset H_1(L;\k).$$

\begin{remark}
The field $\k$ and the ring $Q$ appearing earlier play independent roles in the proof, and need not be the same.
\end{remark}

Consider all affine subspaces parallel to $S_L$; they have the form $S_L+l$
where $l\in H_1(L;\k)$. For each such affine subspace, select all holomorphic
disks among the $\{D^L_i\}$ whose boundary homology class over $\k$ belongs to
that subspace and are non-zero. Also, assume that
the boundaries of the selected disks cancel over $Q$. This cancellation has to
happen in groups for each affine subspace of the form $S_L+l$. The stated
condition can be rewritten as follows:

\begin{equation}
	\label{eq:disk_groups_cancel_bdies}
	\sum_{\mathclap{D_i^L\, :\, [\bd D_i^L]\in S_L+l}}\ \ [\bd D_i^L]=0\in H_1(L;Q)\text{ for each }l\in H_1(L;\k).
\end{equation}
This condition is in general finer than the total cancellation of boundaries (\ref{eq:disk_low_area_cancel_bdies}), and coincides with (\ref{eq:disk_low_area_cancel_bdies}) when we choose $S_L=H_1(L;\k)$.
Under Condition (\ref{eq:disk_groups_cancel_bdies}), we can define $$\OC\clow([p_L],S_L)\subset H_2(X)$$ to be any element whose image under $H_2(X)\to H_2(X,L)$ equals
$$
\sum_{D_i^L\, :\, [\bd D_i^L]\in S_L}[ D_i^L] \in H_2(X,L). 
$$
Note that here we only use the disks whose boundary classes belong to the subspace $S_L\subset H_1(L;\k)$ and ignore the rest. Again, $\OC\clow([p_L],S_L)\subset H_2(X)$ is well-defined up to $[L]$.
The same definitions can be repeated for another Lagrangian submanifold $K$.

\begin{theorem} \label{th:CO_groups} Let $L$ and $K$ be orientable Lagrangian
surfaces and $S_L \subset H_1(L,\k)$, $S_K \subset H_1(L,\k)$ affine subspaces. Assume
that Condition~(\ref{eq:disk_groups_cancel_bdies}) holds for $L, S_L$ and $K,
S_K$, the minima $a$, $A$, $b$, $B$ exist, and $[L]=[K]=0$. Suppose that $a+b<\min(A,B)$ and
the homological intersection number below is non-zero over $Q$:
$$
\OC\clow([p_L],S_L)\cdot \OC\clow([p_K],S_K)\neq 0.
$$
Then $L$ and $K$ are
Hamiltonian non-displaceable.
\end{theorem} 

We point out that, as in Remark~\ref{rmk:L.K=0}, the condition $[L]=[K]=0$ is in fact non-restrictive; however we will not use this observation.

When $L$ or $K$ is monotone, we shall drop the subscript {\it low} from our notation for $\OC\clow(\cdot)$.

\subsection{Computing low-area string invariants}
\label{subsec:compute_low_area}
There is a natural setup for producing Lagrangian submanifolds whose least-area
holomorphic disks will be known. Let us start from an orientable
{\it monotone} Lagrangian $L\subset T^*M$ disjoint from the zero section, and
for which we know the holomorphic Maslov index~2 disks and therefore can compute
our string invariant. We are still restricting to the
4-dimensional setup, so that $\dim M=2$. Next, let us apply fibrewise scaling to
$L$ in order to get a family of monotone Lagrangians $L_a\subset T^*M$ indexed
by the parameter $a\in (0,+\infty)$; we choose the parameter $a$ to be equal to
the areas of Maslov index 2 disks with boundaries on $L_a$. (The scaling changes
the area but not the enumerative geometry of the holomorphic disks.) The next
lemma, explained in Section~\ref{sec:T_a}, follows from an explicit knowledge of
holomorphic disks; recall that we drop the {\it low} subscript from the string
invariants as we are in the monotone setup.

\begin{lemma} 
	\label{lem:cotangent_bundles}
	\begin{enumerate}
		\item[(i)]
		There are monotone Lagrangian tori $\hat L_a\subset T^*S^2$, indexed by $a\in(0,+\infty)$ and called Chekanov-type tori, which bound Maslov index 2 disks of area $a$ and satisfy Condition~(\ref{eq:disk_low_area_cancel_bdies}) over $Q=\Z/4$, such that:
		\begin{equation}
			\label{eq:S2_OC}
			\OC\c([p_{\hat L_a}])=2[S^2]\in H_2(T^*S^2;\Z/4).
		\end{equation}
		Moreover, there is a 1-dimensional affine subspace $S_{\hat L_a}=\langle \beta\rangle\subset H_1(\hat L_a;\Z/2)$
		satisfying Condition~(\ref{eq:disk_groups_cancel_bdies}) over $\k=\Z/2$ and $Q=\Z/2$, such that:
		\begin{equation}
			\label{eq:S2_OC_groups}
			\OC\c([p_{\hat L_a}],S_{\hat L_a})=[S^2]\in H_2(T^*S^2;\Z/2).
		\end{equation}
		
		\item[(ii)]
		Similarly, there are monotone Lagrangian tori $L_a\subset T^*\RP^2$, indexed by $a\in(0,+\infty)$, which bound Maslov index 2 disks of area $a$ and satisfy Condition~(\ref{eq:disk_low_area_cancel_bdies}) over $Q=\Z/8$, such that:
		\begin{equation}
			\label{eq:RP2_OC}
			\OC\c([p_{L_a}])=[4\RP^2]\in H_2(T^*\RP^2;\Z/8).
		\end{equation}
	\end{enumerate}
	In both cases, the tori are pairwise disjoint; they are contained inside any given neighbourhood of the zero-section for small enough $a$.
\end{lemma}

\begin{remark}
	Note that $\RP^2$ is non-orientable so it only has fundamental class over $\Z/2$, 
	however the class $[4\RP^2]$ modulo $8$ also exists.
\end{remark}

Now suppose $M$ itself admits a  Lagrangian embedding 
$M\to X$ into some \emph{monotone} symplectic manifold $X$
(We do not require that $M\subset X$ be monotone.)
By the Weinstein neighbourhood theorem, this embedding extends to a symplectic embedding $i\co U\to X$ for a neighbourhood $U\subset T^*M$ of the zero-section. Possibly by passing to a smaller neighbourhood, we can assume that  $U$ is convex. By construction, the  Lagrangians $L_a$ will belong to $U$ for small enough $a$: 
$$L_a\subset U,\ a\in (0,a_0).$$
(The precise value of $a_0$ depends on the size of the available Weinstein neighbourhood.) We define the Lagrangians

\begin{equation}
	\label{eq:induced_Lags}
	K_a=i(L_a)\subset X, \ a\in (0,a_0)
\end{equation}
which are generally {\it non-monotone} in $X$ (even if $M\subset X$ were monotone).

Consider the induced map 
$
i_*\co H_2(T^*M)\to H_2(X).
$
The next lemma explains that, for sufficiently small $a$, the low-area string
invariants for the $K_a\subset X$ are the $i_*$-images of the ones for the
$L_a\subset T^*M$. We also quantify how small $a$ needs to be.

\begin{lemma}\label{lem:disks_in_nbhood} 
In the above setup, suppose
that the image of the inclusion-induced map $H_1(L;\Z)\to H_1(T^*M;\Z)$ is
$N$-torsion, $N\in\Z$. Let $M\subset (X,\omega)$ be a Lagrangian
embedding into a monotone symplectic manifold $(X,\omega)$. 
 Assume that $\omega$ is scaled in such a way that the area class in
$H^2(X,M;\R)$ is integral, and $c_1(X)=k\omega\in H^2(X;\Z)$ for some $k\in \N$. 
Assume that 
$$
a<1/(k+N).
$$
	
\begin{enumerate}
		\item[(i)] The number $a$ indexing the torus $K_a$ equals the number $a$ 
		defined by Equation~(\ref{eq:defn_a_dim4}). The number $A$ defined by 
		Equation~(\ref{eq: def A}) exists (could be $A = +\infty$) and satisfies: 
		$$A\ge \frac {1-(k-N)a}{N}.$$
		\item[(ii)] There is a tame almost complex structure on $X$ such that all 
		area-$a$ holomorphic Maslov index 2 disks in $X$ with boundary on $K_a$ 
		belong to $i(U)$, and $i$ establishes a 1-1 correspondence between them and 
		 the holomorphic disks in $T^*M$ with boundary on $L_a$. 
	\end{enumerate}
	In particular, when (i) and (ii) apply and $L_a\subset T^*M$ satisfy Condition~(\ref{eq:disk_low_area_cancel_bdies}), the following identity holds in $H_2(X)$:
	$$
	i_*(\OC\c([p_{L_a})])=\OC\clow([p_{K_a}]).
	$$
		Similarly, if
	$L_a\subset T^*M$ satisfy Condition~(\ref{eq:disk_groups_cancel_bdies}) then:
	$$
	i_*(\OC\c([p_{L_a}],S_{L_a}))=\OC\clow([p_{K_a}],S_{K_a}),
	$$
	where $S_{K_a}=i_*(S_{L_a})\subset H_1(K_a;\k)$.
\end{lemma}

A proof is found in Section~\ref{sec:proof_OC}. To give a preview, part~(i) is
purely topological and part~(ii) follows from a neck-stretching argument.

\begin{remark}  The above constructions and proofs
work for any Liouville domain taken instead of $T^*M$. For example, there is
another class of Liouville domains containing interesting monotone Lagrangian
tori: these domains are rational homology balls whose skeleta are the so-called Lagrangian
pinwheels. The embeddings of Lagrangian pinwheels in $\CP^2$ have beed studied
in \cite{ES16}, and using such embeddings we can employ the above construction
and produce non-monotone tori in $\CP^2$ which are possibly
non-displaceable. In the language of almost toric fibrations on $\CP^2$ constructed in
\cite{Vi13}, these tori live above the line segments connecting the baricentre
of a moment triangle to one of the three nodes. See also Subsection~\ref{subsec:higher_dim} for a short discussion of higher dimensions.
\end{remark}

\subsection{Applications to non-displaceability}
Now that we have explicit calculations of the low-area string invariants available, we can start applying our main non-displaceability result. Our first application is to prove Theorem~\ref{th:non_displ_spheres}.

\begin{proof}[Proof of Theorem~\ref{th:non_displ_spheres}] Let $S\subset X$ be a
Lagrangian sphere in a del~Pezzo surface $X$ with an integral symplectic form. 
For concreteness, we normalise the symplectic form to make it primitive integral 
(it integrates by $1$ over some integral homology class). 
Let us define the Lagrangian tori
$T_a = i(\hat L_a) \subset X$ as in (\ref{eq:induced_Lags}), using the monotone tori
$\hat L_a\subset T^*S^2$ which appeared in Lemma~\ref{lem:cotangent_bundles}(i),
and the Lagrangian embedding $S\subset X$. The tori $T_a$ are indexed by
$a\in(0,a_0)$ for some $a_0>0$. Define the tori $T_b'$ indexed by $b\in(0,b_0)$
analogously, using $S'$ instead of $S$. 

After decreasing $a_0$ and $b_0$ if
required, we see that the conditions from Lemma~\ref{lem:disks_in_nbhood}(i,ii)
are satisfied. Here $N = 1$ and $k\in\{1,2,3\}$ depending on the del Pezzo surface, 
by the normalization of our symplectic form. Therefore by Lemma~\ref{lem:disks_in_nbhood} and
Lemma~\ref{lem:cotangent_bundles}(i) we have over $\k=Q=\Z/2$:
	$$
	\OC\clow(T_a,S_{T_a})=[S]\in H_2(X;\Z/2),\quad 
	\OC\clow(T_a',S_{T_a'})=[S']\in H_2(X;\Z/2)
	$$
	for the choices of $S_{T_a}\subset H_1(T_a;\Z/2)$ and $S_{T_a'}$ coming from the
	 one in Lemma~\ref{lem:cotangent_bundles}. Now let us apply Theorem~\ref{th:CO_groups}. We can check the condition $a+b<\min(A,B)$ using
     Lemma~\ref{lem:disks_in_nbhood}(i):
	$$
	a + b < \tfrac{2}{k+N} \le \min(\tfrac{1 - (k-N)a}{N},\tfrac{1 - (k-N)b}{N}) \le \min(A,B)
	$$
	whenever $a,b < \tfrac 1{k+N}=\tfrac1{k+1}$.
	Finally,
	$$
	\OC\clow(T_a,S_{T_a})\cdot 
	\OC\clow(T_a',S_{T_a'})=[S]\cdot[S']=1\in \Z/2.
	$$
So Theorem~\ref{th:CO_groups} implies that $T_a$ is non-displaceable from $T_b'$, for small $a,b$.	
\end{proof}	

 We will prove Theorem~\ref{th:T_a} in Section~\ref{sec:T_a}. In fact, we shall see that the tori $T_a\subset \CP^2$ appearing in Theorem~\ref{th:T_a} can be
 obtained via Formula (\ref{eq:induced_Lags}) using the monotone tori
 $L_a\subset T^*\RP^2$ from Lemma~\ref{lem:cotangent_bundles}(ii), and the
 Lagrangian embedding $\RP^2\subset \CP^2$ described in Section~\ref{sec: dfn
 Tori}.  We could have taken this as a definition, but our actual
exposition in Section~\ref{sec:T_a} is different: we introduce the tori
$T_a\subset \CP^2$ in a more direct and conventional way, and subsequently use
the existing knowledge of holomorphic disks for them to prove
Lemma~\ref{lem:cotangent_bundles}(ii).

\subsection{Higher-dimensional versions}
\label{subsec:higher_dim}
One can state generalisations of Theorems~\ref{th:CO} and~\ref{th:CO_groups} to higher dimensions. We shall omit the precise statements and instead mention a major class of potential examples where the low-area string invariants are easy to define, and which hopefully makes the details clear. The setup is as in Subsection~\ref{subsec:compute_low_area}: one starts with  a monotone Lagrangian submanifold $L$ in a Liouville domain $M$ rescaled so as to stay close to the skeleton of $M$. Then one takes a symplectic embedding $M\subset X$ into some symplectic manifold $X$. 

Consider the composite Lagrangian embedding $L\subset X$ which is not necessarily monotone. Provided that the symplectic form on $X$ is rational and the image of $H_1(L)\to H_1(M)$ is torsion, the classes in $H_2(X,L)$ whose symplectic area is below some treshold (depending on how close we scale $L$ to the skeleton of $M$) lie in the image of $H_2(M,L)$ and therefore \emph{behave as if $L$ were monotone} (meaning that their area is proportional to the Maslov index). This makes it possible to define low-area string invariants for $L\subset X$. An analogue of Lemma~\ref{lem:cotangent_bundles} can easily be established as well; the outcome is that the low-area string invariants of $L$ in $X$ are obtained as the composition
$$
\OC_\low\co 
HF_*(L)\xrightarrow{\OC}H_*(M)\to H_*(X), 
$$
where $HF_*(L)$ is the Floer homology of $L$ in $M$, and $\OC$ is the classical monotone closed-open string map (see Section~\ref{sec:proof_OC} for references). In this setup, it is most convenient to \emph{define} $\OC_\low$ as the above composition. Observe that this setup does not restrict to the use of Maslov index~2 disks, and also allows higher-index disks. 

Let us provide a statement which is similar to Theorem~\ref{th:non_displ_spheres}. 

\begin{theorem}[sketch]
Let $L\subset M$ and $K\subset N$ be  monotone Lagrangian submanifolds in Liouville domains, and these inclusions are $H_1$-torsion. Let $X$ be a monotone symplectic manifold, and $N,M\subset X$ be symplectic embeddings.
Suppose there are classes $\alpha\in HF_*(L)$, $\beta\in HF_*(K)$ such that  the following pairing in $X$ is non-zero:
$$\OC_\low(\alpha)\cdot \OC_\low(\beta)\neq 0.$$
By the above pairing, we mean the composition of the intersection product with the projection: $H_*(X)\otimes H_*(X)\to H_*(X)\to H_0(X)$.

Let $L_a,K_b\subset X$ be the Lagrangians obtained by scaling $L,K$ towards the skeleton within their Liouville domains, and then embedding them into $X$ via $N,M\subset X$.
Then
 $L_a,K_b\subset X$ are Hamiltonian non-displaceable from each other for sufficiently small $a,b$.\qed
\end{theorem}

We remark that the above theorem does not incorporate the modification we made in Theorem~\ref{th:CO}, namely to only consider disks with homologically non-trivial boundary; we also have not explored the possibility to re-organise disks in groups like in Theorem~\ref{th:OCCO_monot}.

The idea of proof is to establish  a version of Lemma~\ref{lem:cotangent_bundles} saying that $\OC_\low$ defined as the above composition coincides with low-area disk counts on the actual Lagrangian $L\subset X$. 
The key observation is that 
for any interval $I=(r,s)\subset  \R$, one can scale $L$ sufficiently close to the skeleton so that all elements in $H_2(X,L)$ of Maslov index within $I$ \emph{and} sufficiently low area, come from $H_2(M,L)$. This can be improved to $(-\infty,s)$ for holomorphic disks by taking $r=-n$, because holomorphic disks of Maslov index $\mu<-n$ generically do not exist.
Given this, one shows that the proof Biran-Cornea's theorem can be carried by only using low-area curves; this forces all the possible bubbling to happen in the monotone regime. The details are left to the reader.

Currently, there seems to be a lack of interesting computations of open-closed maps for monotone Lagrangians in higher-dimensional Liouville domains---this prevents us of from providing some concrete applications of the higher-dimensional story. We believe such applications will become available in the future.

\subsection*{Structure of the article}
In Section~\ref{sec:proof_OC} we prove Theorems~\ref{th:CO}
and~\ref{th:CO_groups}, discuss their connection with the monotone case and
some generalisations. We also prove Lemma~\ref{lem:disks_in_nbhood}.

In Section~\ref{sec:T_a} we prove Lemma~\ref{lem:cotangent_bundles},
Theorem~\ref{th:T_a} and the related theorem for $\CP^1\times\CP^1$ and $\BlIII$. Then we
explain why Floer theory with bulk deformation does not readily apply to
the~$T_a\subset \CP^2$.

\subsection*{Acknowledgements} We are grateful to Paul Biran, Georgios Dimitroglou Rizell, Kenji Fukaya, Yong-Geun~Oh, Kaoru~Ono,
and~Ivan~Smith for useful discussions and their interest in this work; we are also thankful to Denis Auroux, Dusa McDuff and the referees for helping us  spot some  mistakes in the previous version of the article, and suggesting exposition improvements.

Dmitry Tonkonog~is funded by the Cambridge Commonwealth, European and
International Trust, and acknowledges travel funds from King's College,
Cambridge. Part of this paper was written during a term spent at the 
Mittag-Leffler Institute, Sweden, which provided an excellent research environment and financial support for the stay.

Renato Vianna is funded by the Herchel Smith Postdoctoral Fellowship from the 
University of Cambridge.

\section{The non-displaceability theorem and its discussion}
\label{sec:proof_OC}
In this section we prove Theorems~\ref{th:CO} and~\ref{th:CO_groups}, and further discuss them.
We conclude by proving Lemma~\ref{lem:disks_in_nbhood}, which is somewhat unrelated to the rest of the section.

\subsection{The context from usual Floer theory} 
\label{subsec:non_displ_dim4}


We start by explaining Biran-Cornea's non-displaceability criterion for monotone Lagrangians and its relationship with Theorems~\ref{th:CO} and~\ref{th:CO_groups}.
We assume that the reader is familiar with the language of {\it pearly trajectories} to be used here, and shall skip the proofs of some facts we mention if we do not use them later.

Recall that one way of defining the Floer cohomology $HF^*(L)$ of a monotone
Lagrangian $L\subset X$ uses the pearl complex of Biran and Cornea
\cite{BC07,BC09B, BC09A}; its differential counts pearly trajectories consisting of
certain configurations of Morse flowlines on $L$ interrupted by holomorphic
disks with boundary on $L$. A remark about conventions: Biran and Cornea write $QH^*(L)$ instead of $HF^*(L)$; we do not use the Novikov parameter, therefore the gradings are generally defined modulo 2.

Also recall that the basic fact\textemdash if $HF^*(L)\neq 0$, then $L$ is non-displaceable,\textemdash has no intrinsic proof  within the language of pearly
trajectories. Instead, the proof uses the isomorphism relating $HF^*(L)$
to the (historically, more classical) version of Floer cohomology that uses Hamiltonian perturbations. Nevertheless, there is
a different non-displaceability statement whose proof is carried out completely
in the language of holomorphic disks. That statement employs an additional structure, namely the maps $$ \OC\co
HF^*(L)\to QH^*(X),\quad \CO\co QH^*(X)\to HF^*(L) $$ defined by counting
suitable pearly trajectories in the ambient symplectic manifold $X$. These maps
are frequently called the open-closed and the closed-open (string) map,
respectively; note that Biran and Cornea denote them by $i_{L}$, $j_{L}$. The statement we referred to above is the following one.

\begin{theorem}[{\cite[Theorem~2.4.1]{BC09A}}]
	\label{th:OCCO_monot}
	For two monotone Lagrangian submanifolds $L,K$ of a closed monotone symplectic
	manifold $X$, suppose that the composition  
	\begin{equation}
		\label{eq:OCCO}
		HF^*(L)\xrightarrow{\OC} QH^*(X)\xrightarrow{\CO} HF^*(K)
	\end{equation}
	does not vanish. Then  $L$ and $K$ are Hamiltonian non-displaceable. \qed
\end{theorem}

In this paper we restrict ourselves to dimension four, so let us first discuss
the monotone setting of Theorem~\ref{th:OCCO_monot} in this dimension. Recalling
that a del~Pezzo surface has $H_1(X)=0$, we see that there are three
possible ways for (\ref{eq:OCCO}) not to vanish. 

To explain this, it is convenient to pass to chain level: let
$CF^*(L)$ be the Floer chain complex of $L$ (either the pearl complex or the Hamiltonian Floer complex) and $CF^*(X)$ be the Morse (or Hamiltonian Floer) chain complex of $X$, both equipped with Morse $\Z$-gradings. For the definition of the closed-open maps, see \cite{BC07,BC09A} in the pearl setup, and \cite{She13,Ritter2016} among others in the Hamiltonian Floer setup. Below we will stick to the setup with pearls.

First, we can consider the
topological part of (\ref{eq:OCCO}):

\begin{equation*}
 CF^0(L)\xrightarrow[\mu=0]{\OC} CF^2(X)\xrightarrow[\mu=0]{\CO} CF^2(K).
 \end{equation*} 
 In this case, as indicated by the $\mu=0$ labels, the relevant
 string maps necessarily factor through $CF^2(X)$ and are topological,
 i.e.~involve pearly trajectories containing only constant Maslov index 0 disks.
 The composition above computes the homological intersection $[L]\cdot [K]$
 inside $X$, where $[L],[K]\in H_2(X)$. If $[L]\cdot[K] \neq  0$, then $L$ and $K$ are topologically non-displaceable;
 otherwise we proceed to the next possibility;
compare Remark \ref{rmk:L.K=0}.
 Observe that we are using the cohomological grading convention: pearly trajectories of total
 Maslov index $\mu$ contribute to the degree $-\mu$ part of $\CO$, and to the
 degree $(\dim L-\mu)$ part of $\OC$ on cochain level.

The second possibility for $\CO\circ \OC$ not to vanish is via the contribution
of pearly trajectories whose total Maslov index sums to two; the relevant parts
of the string maps factor as shown below. In the examples we are aiming at, we are going to have $[L]=[K]=0$, so
the $\mu=0$ parts below will vanish on homology level.

\begin{equation*}
	\begin{array}{l}
		CF^0(L)\xrightarrow[\mu=0]{\OC}
		CF^2(X)\xrightarrow[\mu=2]{\CO}
		CF^0(K),
		\vspace{0.1cm}
		\\
		CF^2(L)\xrightarrow[\mu=2]{\OC}
		CF^2(X)\xrightarrow[\mu=0]{\CO}
		CF^2(K),
	\end{array}
\end{equation*}

The remaining part of $\CO\circ\OC$ breaks as a sum of three compositions factoring as follows:
\begin{equation}
	\label{eq:OCCO_diamond}
	\xymatrix{
		&
		CF^0(X)\ar_-\cong^{\mu=0}[dr]
		&
		\\
		CF^2(L)\ar^{\mu=2}[r]\ar^{\mu=4}[ur]\ar^-\cong_{\mu=0}[dr]
		&
		CF^2(X)\ar^{\mu=2}[r]
		&
		CF^0(K)
		\\
		&
		CF^4(X)\ar_{\mu=4}[ur]
		&
	}
\end{equation}

The labels here indicate the total Maslov index of holomorphic disks present in
the corresponding pearly trajectories; this time the $\mu=0$ parts are
isomorphisms on homology. Therefore, to compute $\CO\circ\OC|_{HF^2(L)}$ we need to know the
Maslov index~4 disks. We wish to avoid this, keeping in mind that in our
examples we will know only the Maslov index 2 disks.

To this end, we perform the following trick to ``single out'' the Maslov index~2 disk contribution in the diagram above. Let us modify the
definition of $\CO$, $\OC$ by only considering $J$-holomorphic
disks whose boundary is non-zero in $H_1(L;\Z)$ or $H_1(K;\Z)$.
Denote this modified map by $\OC\c$ and consider the composition:
\begin{equation}
	\label{eq:CO_comp_bdy_condition}
	CF^2(L)\xrightarrow[\mu=2]{\OC\c}
	CF^2(X)\xrightarrow[\mu=2]{\CO\c}
	CF^0(K)
\end{equation}
Here the modified maps $\OC\c$, $\CO\c$ by definition count pearly trajectories
contributing to the middle row of (\ref{eq:OCCO_diamond}), i.e.~containing a
single disk, of Maslov index~2, with the additional condition that the boundary
of that disk is homologically non-trivial. The superscript $(2)$ reflects that
we are only considering Maslov index~2 trajectories, ignoring the Maslov index~0
and~4 ones; the condition about non-zero boundaries is not reflected by our
notation. We claim that if
composition~(\ref{eq:CO_comp_bdy_condition}) does not vanish, then $K,L$ are
non-displaceable.

This modified non-displaceability criterion we have just formulated is the specialisation of Theorem~\ref{th:CO} to the case when both Lagrangians are monotone. Indeed, 
if both $L,K$ are monotone and $[L]\cdot [K]=0$, then 
$$\OC\c([p_L])\cdot \OC\c([p_K])\neq 0$$ if and only if the composition (\ref{eq:CO_comp_bdy_condition}) is non-zero; compare Lemma~\ref{lem:configs_CO}.
A proof can also be traced using the original approach \cite[Theorem~8.1]{BC09B}---see the proof of Theorem~\ref{th:CO} in Subsection~\ref{subsec:Proof_ThCO} for further details.

Speaking of our second result, Theorem~\ref{th:CO_groups}, in the monotone case it corresponds to another  refinement of Biran-Cornea's theorem which does not seem to have appeared in the literature. Note that this refinement is not achieved by deforming the Floer theories of $L$ and $K$ by local systems.

\begin{remark} Recall that, for a two-dimensional monotone Lagrangian
	$L$ equipped with the trivial local system, we have 
	$$\sum_j \bd[D_j^L]=0$$ if and
	only if $HF^*(L)\neq 0$, and in the latter case $HF^*(L)\cong H^*(L)$. 
	Indeed, $\sum_j \bd[D_j^L]$
	computes the Poincar\'e dual of the Floer differential $d([p_L])$
	where $[p_L]$ is the generator of $H^2(L)$;
	if we pick a perfect Morse function on $L$, then $p_L$ is geometrically realised by its maximum. 
	If $d([p_L])=0$, then by duality the unit is not hit by the differential, hence $HF^*(L)\neq 0$. For a non-monotone $L$, the condition $\sum_i
	\bd[D_i^L]=0$ from Equation~(\ref{eq:disk_low_area_cancel_bdies}) above  is a natural low-area version of the non-vanishing of Floer cohomology. 
\end{remark}

\begin{remark}
	\label{rem:OC_def_upto}
Recall that $\OC\clow([p_L])\in H_2(X)$ is well-defined up to $[L]$ as explained in Section~\ref{sec:intro}. 
This is analogous to a well known phenomenon in monotone Floer theory: 
 recall that there
 is no canonical identification between $HF^*(L)$ and $H^*(L)$, even when they
 are abstractly isomorphic \cite[Section~4.5]{BC09A}. In particular, $HF^*(L)$
 is only $\Z/2$-graded and the element $[p_L]\in HF^*(L)$ corresponding to the
 degree 2 generator of $H^2(L)$ is defined up to adding a multiple of the unit
 $1_L\in HF^*(L)$. Recall that $\OC(1_L)$ is dual to $[L]\in H_2(X)$, and this
 matches with the fact that $\OC([p_L])$, as well as $\OC\c([p_L])$, is defined
 up to $[L]$. \end{remark}

\begin{remark}
	Charette~\cite{Cha15B} defined quantum Reidemeister torsion for monotone
	Lagrangians whose Floer cohomology vanishes. While it is possible that his
	definition generalises to the non-monotone setting, making our tori
	$T_a\subset \CP^2$ valid candidates as far as classical Floer theory is concerned, 
	it is shown in \cite[Corollary 4.1.2]{Cha15B} that quantum
	Reidemeister torsion is always trivial for tori. 
\end{remark}

\subsection{Proof of Theorem~\ref{th:CO}} \label{subsec:Proof_ThCO}
Our proof essentially follows~\cite[Theorem~8.1]{BC09B}  with the following
differences: we check that certain unwanted bubbling, impossible in the monotone
case, does not occur in our setting given that $a+b< \min(A,B)$; we include an
argument which ``singles out'' the contribution of Maslov index~2 disks with
non-trivial boundary from that of Maslov index~4 disks; and relate the string
invariants $\OC\clow([p_L])$, $\OC\clow([p_K])$ defined in
Section~\ref{sec:intro} to the ones appearing more naturally in pearly
trajectory setup. We assume that the reader is familiar with the setup moduli spaces of pearly trajectories \cite{BC09B}.

\begin{remark}
 We point out that  \cite[Theorem 8.1]{BC09B} also appears as \cite[Theorem~2.4.1]{BC09A}, and in the latter reference the authors take a different approach to a proof based on superheaviness.
\end{remark}

\begin{lemma}
	\label{lem:OC_invt}
Under the assumptions of Subsection \ref{subsec: DefOC_low}, the string invariants 
\linebreak[4]
$\OC\clow([p_L])$ and $\OC\clow([p_L],S_L)$ are independent of 
the choice of $J$ and the marked point $p_L$.  
\end{lemma}

\begin{proof}
First, we claim that for a generic 1-parametric family of almost complex
structures, $L$ will not bound
holomorphic disks of Maslov index $\mu<0$. Indeed, for simple disks this follows
for index reasons (recall that $\dim X=4$); next, non-simple disks with $\mu<0$
must have an underlying simple disk with $\mu<0$ by the decomposition theorem of
Kwon and Oh \cite{KO00} and Lazzarini \cite{La00}, so the non-simple ones do not
occur as well.

Therefore, the only way disks with $\mu=2$ and area $a$ can bubble is into a
stable disk consisting of $\mu=0$ and $\mu=2$ disks; the latter $\mu=2$ disk
must have positive area less than $a$. However, such $\mu=2$ disks do not exist
by Condition~(\ref{eq:defn_a_dim4}). We conclude that Maslov index~2, area $a$
disks cannot bubble as we change $J$, and because the string invariants are
defined in terms of these disks, they indeed do not change. 
A similar argument (by moving the point along a generic path) shows the independence of $p_L$.
\end{proof}	

\begin{remark} \label{rmk:LemSimilarto}
 In the monotone case, the fact that the counts of Maslov index~2 holomorphic disks 
 was first pointed out in \cite{ElPo93}. A rigorous proof uses the
 works \cite{KO00,La00} as above, and is well known. Similar results for
 (possibly) non-monotone Lagrangians in dimension~4 appear in \cite[Lemmas~2.2,
 2.3]{Cha15A}. 
 \end{remark}

Suppose there exists a Hamiltonian diffeomorphism $\phi$ such that $\phi(L)\cap
K=\emptyset$, and redenote $\phi(L)$ by $L$, so that $L\cap K=\emptyset$.


Pick generic metrics and Morse functions $f_1,f_2$ on $L,K$. We assume that the
functions $f_1,f_2$ are perfect (it simplifies the proof, but is not essential);
such exist because $L,K$ are two-dimensional and orientable.  Consider the
moduli space $\M$ of configurations (``pearly trajectories'') of the three types
shown in Figure~\ref{fig:def_Moduli}, with the additional condition that the
total boundary homology classes of these configurations are non-zero both in
$H_1(L;\Z)$ and $H_1(K;\Z)$.  (By writing ``total'' we mean that if the
configuration's boundary on a single Lagrangian has two components, their sum
must be non-zero.)

\begin{figure}[h]
	\includegraphics{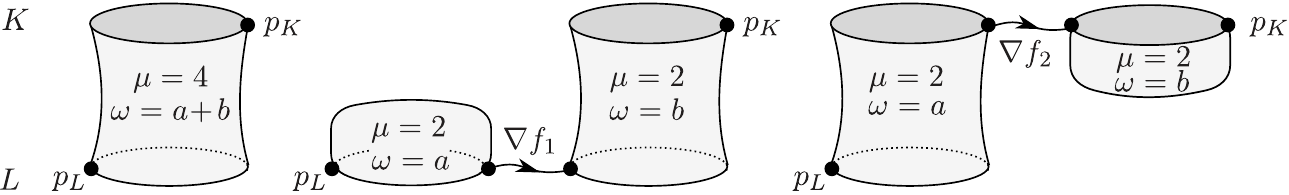}
	\caption{The moduli space $\M$ consists of pearly trajectories of these types.}
	\label{fig:def_Moduli}
\end{figure}

Here are the details on the pearly trajectories from Figure~1 that we use to define $\M$:
\begin{itemize}
 \item The Maslov index and the
area of each curve is prescribed in the figure;
\item  The conformal parameter of each annulus is allowed to take
any value $R\in(0,+\infty)$. Recall that the domain of an annulus with conformal parameter $R$ can be realised as
$\{z\in \C:1\le |z|\le e^R\}$;
\item Every flowline has a time-length parameter $l$
that can take any value $l\in[0,+\infty)$. 
The configurations with a contracted flowline (i.e.~one with $l=0$) correspond to interior points of $\M$, because gluing the disk to
the annulus is identified with $l$ becoming negative;
\item The annulus has two marked points, one on each boundary component, that are \emph{fixed in the domain}. This means that if we identify an annulus with conformal parameter $R$ with $\{z\in \C:1\le |z|\le e^R\}$, then the two marked points can be e.g.~$1$ and $e^R$;
\item The disks also have marked points as shown in the figure. Because the disks are considered up to reparametrisation, the marked points can also be assumed to be fixed in the domain;
\item  The curves evaluate to the fixed points
fixed points $p_K\in K$, $p_L\in L$ at the marked points as shown in Figure~\ref{fig:def_Moduli}. They satisfy the Lagrangian boundary conditions as shown  in Figure~\ref{fig:def_Moduli};
\item  The non-vanishing of total boundary homology classes (stated above) must be satisfied.
\end{itemize}

Recall that the Fredholm index of unparametrised holomorphic annuli without
marked points and with free conformal parameter equals the Maslov index.
Computing the rest of the indices, one shows that $\M$ is a smooth 1-dimensional
oriented manifold \cite[Section~8.2]{BC09B}, assuming $\M$ is regular. The
regularity of the annuli can be achieved by a small domain-dependent
perturbation of the $J$-holomorphic equation;  we give a detailed discussion of it in the
next subsection. Now, we continue with the proof assuming the regularity of the annuli (and hence of $\M$, because the transversality for disks is classical).


The space $\M$ can be compactified by adding configurations with broken
flowlines as well as configurations corresponding to the conformal parameter $R$ of the annulus becoming $0$ or $+\infty$.
We describe each of the three types of configurations separately and determine
their signed count.

{\it (i)} The configurations with broken flowlines are shown in Figure~\ref{Fig:M_Morse_break}. As before, they are subject to the condition that the total boundary homology classes of the configuration are non-zero both in $H_1(L;\Z)$ and $H_1(K;\Z)$. The annuli
have a certain conformal parameter $R_0$ and the breaking is an index 1 critical
point of $f_i$ \cite[Section~8.2.1, Item (a)]{BC09B}.

\begin{figure}[h]
	\includegraphics{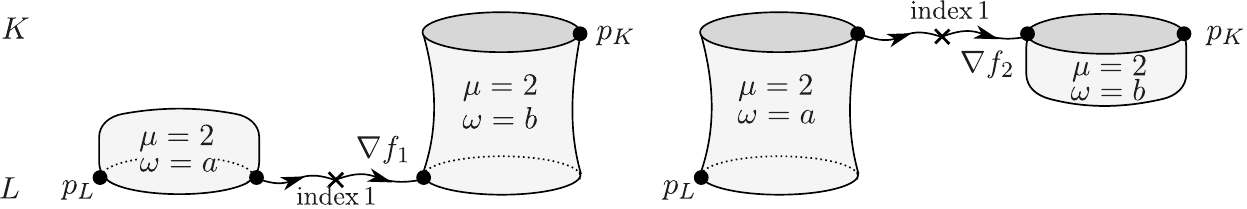}
	\caption{Configurations with broken flowlines, called type (i).}
	\label{Fig:M_Morse_break} 
\end{figure}

\noindent The count of the sub-configurations consisting of the disk and the
attached flowline vanishes: this is a Morse-theoretic restatement
Condition~(\ref{eq:disk_low_area_cancel_bdies}) saying that $\sum_i \bd
[D_i^L]=\sum_j \bd[D_j^K]=0$. Hence (by the perfectness of the $f_i$) the count of
the whole configurations in Figure~\ref{Fig:M_Morse_break} also vanishes, at
least if we ignore the condition of non-zero total boundary. Separately, the
count of configurations in Figure~\ref{Fig:M_Morse_break} whose total boundary
homology class is zero either in $L$ or $K$, also vanishes. Indeed, suppose for
example that the $\omega=a$ disk in Figure~\ref{Fig:M_Morse_break} (left) has
boundary homology class  $\l \in H_1(L;\Z)$ and the lower boundary of the annulus
has class $-\l$; then the count of the configurations in the figure with that
disk and that annulus equals the homological intersection $(-\l)\cdot \l=0$,
since $L$ is an oriented surface.  We
conclude that the count of configurations in the above figure whose total
boundary homology classes are {\it non-zero}, also vanishes.

{\it (ii)} The configurations with $R=0$ contain a curve whose domain is an
annulus with a contracted path connecting the two boundary components. The
singular point of this domain must be mapped to an intersection point $K\cap L$,
so these configurations do not exist if $K\cap L=\emptyset$ \cite[Section~8.2.1, Item (c)]{BC09B}.  

{\it (iii)} The configurations with $R=+\infty$ correspond to an annulus
breaking into two disks, one with boundary on $K$  and the other with boundary on
$L$ \cite[Section~8.2.1, Item (d)]{BC09B}. One of the disks can be constant, and the possible configurations are shown
in Figure~\ref{fig:config_R_infty}.

\begin{figure}[h]
	\includegraphics{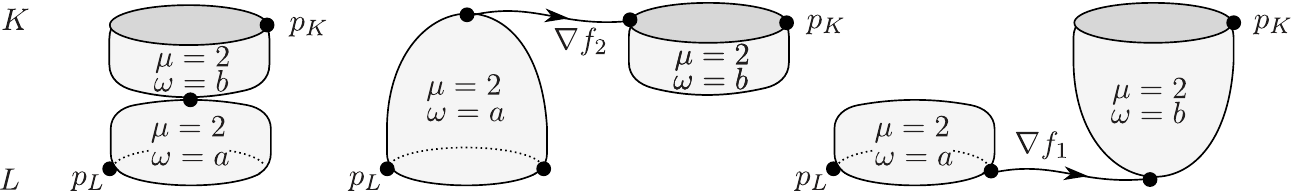}
	\caption{The limiting configurations when $R=+\infty$, called type (iii).}
	\label{fig:config_R_infty}
\end{figure}

In fact, there is another potential annulus breaking at $R=+\infty$ that we have
ignored: the one into a Maslov index 4 disk on one Lagrangian and a (necessarily
constant) Maslov index 0 disk on the other Lagrangian, see
Figure~\ref{fig:config_R_infty_unwanted}. This broken configurations cannot
arise from the configurations in $\M$ by the non-zero boundary condition imposed
on the elements of this moduli space. The fact that a Maslov index 0 disk has to
be constant is due to the generic choice of $J$.

\begin{figure}
	\includegraphics{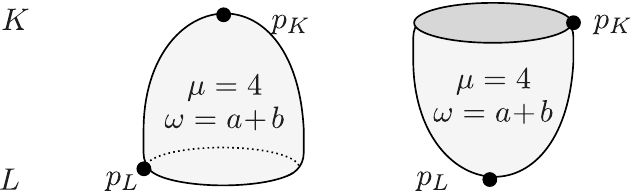}
	\caption{The limiting configurations for $R=+\infty$ which are impossible by the non-zero boundary condition.}
	\label{fig:config_R_infty_unwanted}
\end{figure}

\begin{lemma}
	\label{lem:configs_CO}
	The count of configurations in Figure~\ref{fig:config_R_infty} equals $\OC\clow([p_L])\cdot \OC\c([p_K])$ as defined in Section~\ref{sec:intro}. 
\end{lemma}

\begin{proof} In the right-most configuration in
Figure~\ref{fig:config_R_infty}, forget the $\omega=b$ disk so that one endpoint
of the $\nabla f_1$-flowline becomes free; let $C^L$ be the singular 2-chain on
$L$ swept by these endpoints.  In other words, for each disk $D_i^L$, consider the
closure of 
$$\CC_i^L = \{ \phi_l(x) \in L\ :\  x \in \bd D_i^L,\ \phi_l \ \text{is the time-}l \ \text{flow of}
\ \nabla f_1,\ l \ge 0 \}$$ oriented so that the component of $\bd \CC_i^L$ corresponding
to $\bd D_i^L$ has the same orientation as $\bd D_i^L$. Then $C^L = \bigcup_i 
\overline{\CC_i^L}$.

We claim that $\bd C^L=\sum_{i}\bd
D_i^L$ on chain level. Indeed, the boundary $\bd C^L$ corresponds to zero-length
flowlines that sweep $\sum_{i}\bd D_i^L$, and to flowlines broken at an index
1 critical point of $f_1$, shown below:
	\vspace{0.1cm}
	
	\noindent
	\hspace*{\fill}\includegraphics{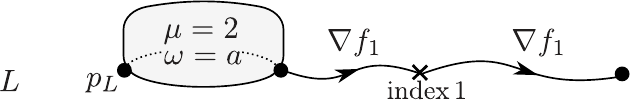}
	\hspace*{\fill}
	\vspace{0.1cm}
	
	\noindent
	The endpoints of these configurations sweep the zero 1-chain. Indeed, we are given that
	$\sum_{i}\bd[D_i^L]=0$ so the algebraic count of the appearing index~1 critical points represents a null-cohomologous Morse cocycle, therefore this count equals zero by perfectness of $f_1$. It follows that $\bd C^L=\sum_{i}\bd
	D_i^L$.
 Similarly, define the 2-chain $C^K$ on $K$, $\bd C^K=\sum_j\bd D_j^K$, by
 forgetting the $\omega=a$ disk in the second configuration of type (iii) above,
 and repeating the construction. It follows that the homology class
 $\OC\clow([p_L])$ from Subsection~\ref{subsec: DefOC_low} can be represented by the cycle 
 $$(\cup_i D_i^L)\cup C^L,$$ and
 similarly $\OC\c([p_K])$ can be represented by $(\cup_j D_j^K)\cup C^K$. 
 This intersection number can be expanded into four chain-level
 intersections: 
 $$ 
 \OC\clow([p_L])\cdot \OC\clow([p_K])=(\cup_i D_i^L)\cdot (\cup_j
 D_j^K)+(\cup_i D_i^L)\cdot C^K+ C^L\cdot(\cup_j D_j^K)+C^L\cdot C^K. 
 $$ 
 The last summand vanishes because $L\cap K=\emptyset$, and the other summands
 correspond to the three configurations of type (iii) pictured earlier. 
\end{proof} 

\begin{remark}	
 Note that the equality between the intersection number $\OC\clow([p_L])\cdot
 \OC\clow([p_K])$ and the count of the $R=+\infty$ boundary points of $\M$ holds
 integrally, i.e.~with signs. This follows from the general set-up of
 orientations of moduli spaces in Floer theory, which are consistent with taking
 fibre products and subsequent gluings, see e.g.~\cite[Appendix C]{Ab10} for the case most relevant to us. For example, in our case the signed
 intersection points between a pair of holomorphic disks can be seen as the
 result of taking fibre product along evaluations at interior marked points;
 therefore these intersection signs agree with the orientations on the moduli
 space of the glued annuli. 
 \end{remark}

If the moduli space $\M$ is completed by the above configurations (i)---(iii),
it becomes compact. Indeed, by the condition $a+b< \min(A,B)$, Maslov index 2
disks on $L$ with area higher than $a$ cannot bubble. Disks of Maslov index
$\mu\ge 4$ cannot bubble (for finite $R$) on either Lagrangian because the rest
of the configuration would contain an annulus of Maslov index $\mu \le 0$
passing through a fixed point on the Lagrangian, and such configurations have
too low index to exist generically (the annuli can be equipped with a generic
domain-dependent perturbation of $J$, hence are regular). Similarly, holomorphic
disks of Maslov index $\mu\le 0$ cannot bubble as they do not exist for generic
perturbations of the initial almost complex structure $J$. (This is true for
simple disks by the index formula, and follows for non-simple ones from the
decomposition theorems \cite{La11,KO00}, as such disks must have an underlying
simple disk with $\mu\le 0$.) Side bubbles of Maslov index 2 disks (not
carrying a marked point with a $p_K$ or a $p_L$ constraint) cannot occur because
the remaining Maslov index 2 annulus, with both the $p_K$ and $p_L$ constraints,
would not exist generically. Finally, as usual, sphere bubbles cannot happen in a
1-dimensional moduli space because such bubbling is a codimension 2 phenomenon in the monotone case.

By the compactness of $\M$, the signed count of its boundary points (i)---(iii)
equals zero. We therefore conclude from Lemma~\ref{lem:configs_CO} and the
preceding discussion that $\OC\clow([p_L])\cdot \OC\c([p_K])=0$, which contradicts the
hypothesis of Theorem~\ref{th:CO}.\qed

\subsection{Transversality for the  annuli}
\label{subsec:trans}
In the proof of
Theorem~\ref{th:CO}, we mentioned that in order to make the annuli appearing in the moduli space $\M$
regular, we need to use a domain-dependent perturbation of the $J$-holomorphic
equation on those annuli. We wish to make this detail explicit. 


First, let us recall the moduli space of domains used to define $\M$, see Figure~\ref{fig:def_Moduli}, and its compactification. Recall that the annuli in $\M$ were allowed to have free conformal parameter 
$$R\in(0,+\infty),$$
and the limiting cases $R=0$, $R=+\infty$ were included in the compactifcation of $\M$  as explained in the proof above.

\begin{figure}[h]
	\includegraphics[]{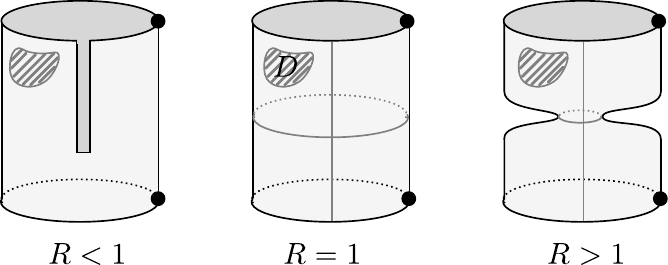}
	\caption{The annuli $A_R$ with various conformal parameters, and a compact disk $D=D_R$ supporting the domain-dependent almost complex structure $J_D$.}
	\label{fig:perturb_annuli}
\end{figure}
Now, for each conformal parameter $R\in(0,+\infty)$, we pick some closed disk inside the corresponding annulus:
$$
D_R\subset A_R=\{1\le |z|\le e^R\}
$$
depending smoothly on $R$, and disjoint from the regions where the domains $A_R$ degenerate as $R\to+\infty$ or $R\to 0$. To see what the last condition means, recall that the annulus $A_R$ used for defining $\M$ has two fixed boundary marked points; we are assuming they are the points $1$ and $e^R$. With these marked points, the annulus has no holomorphic automorphims. We can then uniformise the family of annuli $\{A_R\}_{R\in(0,\infty)}$ as shown in Figure~\ref{fig:perturb_annuli}. To do so, we start with an annulus of a fixed conformal parameter, say $R=1$. The annuli with $R<1$ are obtained from the $R=1$ annulus by performing a slit along a fixed line segment $C\subset A_{1}$ connecting the two boundary components. The annuli for $R>1$ are obtained by stretching the conformal structure of the $R=1$ annulus in a fixed neighbourhood of some core circle $S\subset A_{1}$.  In this presentation, all annuli $A_{R}$ have a common piece of domain away from a neighbourhood of $C\cup S$, and it is notationally convenient to choose $D_R\subset A_R$ to be a fixed closed disk $D$ inside that common domain, for each $R$. See Figure~\ref{fig:perturb_annuli}.

Next, let $\J$ denote the space of compatible almost complex structures on $X$. Let $J\in \J$ be the almost complex structure  we have been using in the proof of Theorem~\ref{th:CO}---namely, we are given that the relevant $J$-holomorphic disks are regular.

 Now pick some domain-dependent almost complex structure
$$
J_{A_R}\in C^{\infty} (A_R,\J),\quad J_{A_R}\equiv J \text{ away from }D_R.
$$
Using the above presentation where all the $D_R$ are identified with a fixed disk $D$, it is convenient to take $J_{A_R}|_{D_R}$ to be the same domain-dependent almost complex structure 
$$J_D\in C^\infty(D,\J)$$ for all $R$, 
such that $J_D\equiv J$ near $\bd D$.

In our (modified) definition of $\M$, we use the following domain-dependent perturbation of the $J$-holomorphic equation for each of the appearing annuli:
\begin{equation}
\label{eq:perturbed}
du+J_{A_R}(u)\circ du\circ j=0,\quad u\co A_R\to X
\end{equation} 
Here $j$ is the complex structure on $A_R$.
Observe that
(\ref{eq:perturbed}) restricts to the usual $J$-holomorphic equation, with the
constant $J$, away from $D_R$. For this reason, the equation is $J$-holomorphic
in the neighbourhoods of the nodal points formed by: domain degenerations as
$R\to +\infty$ and $R\to 0$, and the side bubbling of holomorphic disks. The
standard gluing and compactness arguments work in this setting, compare
e.g.~Sheridan~\cite[Proof of Proposition~4.3]{She11}; therefore, the $R=0$ and
$R=+\infty$ compactifications of $\M$ from the proof of Theorem~\ref{th:CO} are
still valid in our perturbed case, as well as the fact that the flowline length
0 configurations in Figure~\ref{fig:def_Moduli}~(middle and right) are interior
points of $\M$. (This is a very hands-on version of the general notion of a
consistent choice of perturbation data in the setup of by
Seidel~\cite{SeiBook08}, see also~\cite{She11}, except that we are not using a
Hamiltonian term in our equation.) We keep all disks appearing in $\M$ to be $J$-holomorphic, without any perturbation.

It is well known that the solutions to (\ref{eq:perturbed}) are regular for a Baire set of $J_{A_R}|_{D_R}$, or
equivalently for a Baire set of $J_D$'s \cite[Lemma~4.1, Corollary~4.4]{CM07}. The fact that a perturbation in a neighbourhood of a point is sufficient follows from the unique continuation principle for Cauchy-Riemann type equations and is contained in the statement of \cite[Lemma~4.1]{CM07}.

In particular, there is a sequence $J_{A_R,n}$ converging
to the constant $J$:
$$
J_{A_R,n}|_{D_R}\to J \quad \text{ in } C^{\infty}(D_R,\J)
$$
such that the annuli solving (\ref{eq:perturbed}) with respect to $J_{A_R,n}$ are regular for each $n$. 

The rest of the proof of Theorem~\ref{th:CO} requires one more minor modification.  Looking
at Figure~\ref{fig:config_R_infty}~(left), one of the holomorphic disks now
carries a domain-dependent perturbation supported in the subdomain $D$ inherited
from the annulus, compare Figure~\ref{fig:perturb_annuli}~(right). (Which of the
two disks carries this perturbation depends on which side of the core circle the
$D$ was in the annulus.) We note that the disks in
Figure~\ref{fig:config_R_infty}~(middle and right) do not carry a perturbation,
as they arise as side bubbles from the annuli. We claim that for large enough
$n$, the count of the configurations in Figure~\ref{fig:config_R_infty}~(left)
where one of the disks carries an above perturbation, equals to the same count
where the disks are purely $J$-holomorphic. Indeed, it follows from the fact
that $J_{D,n}\to J$, using continuity and the regularity of the unperturbed
$J$-holomorphic disks. This allows us to use Lemma~\ref{lem:configs_CO} and
conclude the proof.  

\subsection{Proof of Theorem~\ref{th:CO_groups}}
This is a simple modification of the proof of Theorem~\ref{th:CO}, so we shall be brief. The idea is to redefine the moduli space $\M$ by considering only those configurations in Figure~\ref{fig:def_Moduli} whose total boundary homology classes in $H_1(L;\k)$ resp.~$H_1(K;\k)$ belong to the affine subspace $S_L$ resp.~$S_K$. 

The only difference in the proof arises when we argue that configurations of type~(i) cancel, see Figure~\ref{Fig:M_Morse_break}. At that point of the above proof, we used Condition~(\ref{eq:disk_low_area_cancel_bdies}); now we need to use Condition~(\ref{eq:disk_groups_cancel_bdies}) instead. 
Let us consider configurations as in the left part of Figure~\ref{Fig:M_Morse_break}. Assume that the area $b$ annulus in the figure has boundary homology class $l\in H_1(L;\k)$ on $L$. Then the area $a$ disk of the same configuration has boundary class belonging to the affine subspace $S_L-l\subset H_1(L;\k)$; this is true because the total boundary homology class has to lie in $S_L$. By a Morse-theoretic version of Condition~(\ref{eq:disk_groups_cancel_bdies}), the count of such area $a$ disks with the attached flowlines (asymptotic to index 1 critical points) vanishes. The rest of the proof goes without change.\qed

\subsection{Adjusting the area conditions}\label{subsec:wall_crossing}
We would like to point out that  the area restrictions in Theorems~\ref{th:CO} and
~\ref{th:CO_groups} can be weakened at the expense of requiring one of the two
Lagrangians be monotone. In the setup of Section~\ref{subsec: DefOC_low}, suppose that $K$ is monotone, so that $B = +\infty$ and $b$ is the
monotonicity constant of $X$, i.e. $\omega(\beta) = \frac{b}{2} \mu(\beta)$ for
$\beta \in H_2(X,K)$. 
Below is a modified version of Theorem~\ref{th:CO}, which differs by the fact that the numbers $a,A$ become redefined compared to those from Section~\ref{subsec: DefOC_low}. We are still using a coefficient ring $Q$ for homology.

\begin{theorem}
	\label{thm:alg_cancel}
Suppose $K,L\subset X$ are orientable Lagrangian surfaces, and $K$ is monotone.
Fix any tame almost complex structure $J$, and let $\M_C(pt)$ be the moduli space of holomorphic disks in the homology class $C$ through a fixed point in $L$. Define $A$ to be
\begin{equation}
A=\min\left\{\omega_0:
\sum_{
	\begin{smallmatrix}
	{C\in H_2(X,L):}
	\\ \omega(C)=\omega_0,\
	\mu(C)=2
	\end{smallmatrix}
	}\quad \sum_{u\in \M_C(pt)}[\bd u] \neq 0\in H_1(L)\right\}
\end{equation}
(This minimum exists by Gromov compactness.)
Let $a$ be any number less than $A$, and assume in addition that that all holomorphic disks of area less than $a+b$ with boundary on $L$ are regular with respect to the chosen $J$. Define $\OC\clow([p_L])$ as in Section~\ref{subsec: DefOC_low} using holomorphic disks of area $a$.
If $a+b<A$, $[L]=[K]=0$, and
$$
\OC\clow([p_L])\cdot \OC\c([p_K])\neq 0,
$$
then $L$ and $K$ are Hamiltonian non-displaceable from each other.
\end{theorem}

The meaning of this modification is as follows. In Theorem~\ref{th:CO}, $a$ and $A$ were the first two positive areas of classes in $H_2(X,L)$; moreover the boundares of area-$a$ disks had to cancel, in order for the string invariant to be defined. Here the boundaries of area-$a$ disks still cancel by the new definition of $A$ and because $a<A$.

Now we come to the difference: in the setup of Theorem~\ref{th:CO}, there existed no (topological) disks with areas between $a$ and $a+b$, while in the setup of Theorem~\ref{thm:alg_cancel}, such disks could exist (and be holomorphic of Maslov index 2), but their boundaries will still cancel by the assumption $a+b<A$. It turns out that this is sufficient to run the argument.

We can similarly modify Theorem~\ref{th:CO_groups}. 
\begin{theorem}
\label{thm:alg_cancel_groups}
In the setup of the previous theorem,
additionally choose $S_L$ as in Section~\ref{sec:intro}. The statement of
Theorem~\ref{thm:alg_cancel} holds if we replace
$A$ 
by
\begin{equation}
 A=\min\left\{\omega_0:
 \sum_{
 	\begin{smallmatrix}
 	{C\in H_2(X,L):}\\
 	\omega(C)=\omega_0,\ \mu(C)=2,\\
 	\bd C\in S_L+l
 	\end{smallmatrix}
 }\quad \sum_{u\in \M_C(pt)}[\bd u] \neq 0\in H_1(L) \text{ for some }l\in H_1(L)\right\}
\end{equation}
and $\OC\clow([p_L])$ by $\OC\clow([p_L],S_L)$.
\end{theorem}

\begin{proof}[Proof of Theorems~\ref{thm:alg_cancel},~\ref{thm:alg_cancel_groups}]
Given that $K$ is monotone,  $\OC\c([p_K])$ is obviously
invariant under its Hamiltonian isotopies. After this, the  proofs of Theorems~\ref{th:CO} and~\ref{th:CO_groups} can be repeated using the fixed $J$ appearing in the hypothesis, with one obvious adjustment: the configurations in Figures~\ref{fig:def_Moduli} and~\ref{Fig:M_Morse_break} must allow disks of any area less than $a+b$ with boundary on $L$. The configurations in Figure~\ref{Fig:M_Morse_break} still cancel by hypothesis. Note that no new configurations of type~(iii) (see Figure~\ref{fig:config_R_infty}) need to be included.

In order to achieve the transversality of annuli by the method of Subsection~\ref{subsec:trans}, we choose domain-dependent perturbations $J_{A_R,n}$ that converge, as $n\to+\infty$, to the fixed $J$ appearing in the hypothesis.
\end{proof}

\subsection{Proof of Lemma~\ref{lem:disks_in_nbhood}}

We start with Part~(i) which is purely topological.
Assuming that $H_1(L;\Z)\to H_1(T^*M;\Z)$ is $N$-torsion, for any class in $D\in H_2(X,T_a;\Z)$ its $N$-multiple can be written in the following general form:
$$
ND=i_*(D')+D'',\quad D'\in H_2(T^*M,L_a;\Z),\quad D''\in H_2(X;\Z).
$$
Recall that $\omega=c_1/k\in H^2(X;\R)$ is integral. Assuming $\mu(D)=2$, we compute:
$$
\begin{array}{lll}
\mu(ND)&=&\mu(D')+2c_1(D'')=2N,\\
\omega(ND)&=&a\cdot \mu(D')/2+c_1(D'')/k\\
&=&a(N-c_1(D''))+c_1(D'')/k\in \{aN+(1-ka)\Z\}.
\end{array}
$$
Above, we have used the fact that the $L_a$ are monotone in $T^*M$, and  that $c_1(D'')$ is divisible by $k$.
Therefore,
$$
\omega(D)\in \{a+\textstyle\frac 1 N (1-ka)\Z\}
$$
When $a<1/(k+N)$, the least positive number in the set $\{a+\textstyle\frac 1 N
(1-ka)\Z\}$ is $a$, and the next one is $A=a+\frac 1 N (1-ka)$. This proves
Lemma~\ref{lem:disks_in_nbhood}(i). Notice that area $a$ is achieved if and only
if $c_1(D'')=\omega(D'')=0$. 

To prove Lemma~\ref{lem:disks_in_nbhood}(ii), first notice that holomorphic
disks with boundary on $L_a\subset T^*M$ must be contained in $U\subset T^*M$ by
the maximum principle, for any almost complex structure cylindrical near $\bd
U$. Therefore to prove the desired 1-1 correspondence between the holomorphic
disks, it suffices to prove that for some almost complex $J$ on $X$, the
area-$a$ Maslov index 2 holomorphic disks with boundary on $T_a$ are contained
in $i(U)$. We claim that this is true for a $J$ which is sufficiently stretched
around $\bd i(U)$, in the sense of SFT neck-stretching.

\begin{figure}
	\includegraphics[]{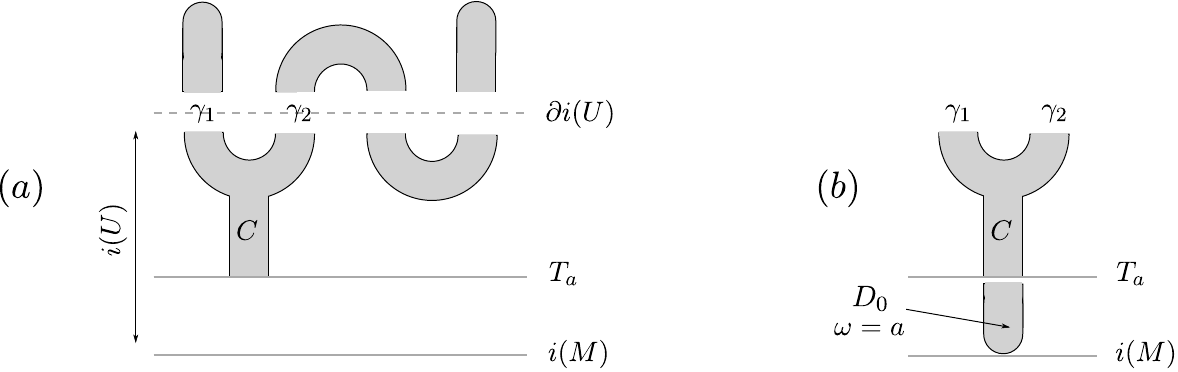}
	\caption{(a): holomorphic building which is the limit of a holomorphic disk, and its part $C$; (b): the area computation for $C$.}
	\label{fig:neck_stretch}
\end{figure}
Pick the standard Liouville 1-form $
\theta$ on $i(U)$, and stretch $J$ using a cylindrical almost complex structure with respect to $\theta$ near $\bd U$.
 The SFT compactness theorem \cite{CompSFT03} implies that disks not contained
 in $i(U)$ converge in the neck-stretching limit to a holomorphic building, like
 the one shown in Figure~\ref{fig:neck_stretch}(a). One part of the building is
 a curve with boundary on $T_a$ and several punctures. Denote this curve by $C$.
 It is contained in $i(U)$, and its punctures are asymptotic to Reeb orbits in
 $\bd i(U)$ which we denote by $\{\gamma_j\}$.
 	
Recall that above we have shown  that the homology class of the original disk $D$ had the form $$D=i_*(D')/N+D''/N\in H_2(X,T_a;\Q),$$ where $D''$ is a closed 2-cycle and $\omega(D'')=0$.
	Denote 
	$$ D_0=i_*(D')/N\in H_2(X,T_a;\Q).$$ 
	Then $\omega( D_0)=a$ and $ D_0$ can be realised as a chain sitting inside $i(U)$, whose boundary in $T_a$ matches the one of $C$ (or equivalently, $D$). 
	Consider the chain $C\cup (-D_0)$, where $(-D_0)$ is the chain $D_0$ taken with the opposite orientation, see Figure~\ref{fig:neck_stretch}(b). Then: 
	
 $$\bd \left(C\cup (-D_0)\right)=\cup_j\gamma_j.$$ 
 Below, the second equality
 follows from the Stokes formula using $\omega = d\theta$, which can be applied because the whole chain is contained in
 $i(U)$: 
 $$\omega(C)-a=\omega(C\cup (-D_0))=\textstyle \sum_j A(\gamma_j),$$
 where 
 $$ A(\gamma_j) = \textstyle \int_{\gamma_j} \theta > 0,$$ 
 since $\theta( \textit{Reeb vector field}) = 1$.  
 
 On the other hand, recall that $\omega(C)<a$ because $C$ is part of a
 holomorphic building with total area $a$. This gives a contradiction. We
 conclude that all area-$a$ Maslov index 2 holomorphic disks are contained in
 $i(U)$ for a sufficiently neck-stretched $J$.\qed

\section{The tori $T_a$ are non-displaceable from the Clifford torus} \label{sec:T_a}

In this section we recall the definition of the tori $\T_a\subset
\CP^1\times\CP^1$ which were studied by Fukaya, Oh, Ohta and Ono \cite{FO312},
and the tori $T_a\subset \CP^2$ appearing in the introduction. We prove
Theorem~\ref{th:T_a} along with a similar result for the $\T_a \subset \CP^1
\times \CP^1$, and for an analogous family of tori in the 3-point blowup of $\CP^2$.
We also prove Lemma~\ref{lem:cotangent_bundles}, and check that Floer cohomology
with bulk deformations vanishes for the $T_a$. 

\subsection{Definition of the tori.} \label{sec: dfn Tori}

We choose to define the tori $T_a$ as in \cite{Wu15}, using the {\it coupled
spin} system \cite[Example 6.2.4]{PR14} on $\CP^1\times\CP^1$. Consider
$\CP^1\times \CP^1$ as the double pendulum composed of two unit length rods: the
endpoint of the first rod is attached to the origin $0\in\R^3$ around which the
rod can freely rotate; the second rod is attached to the other endpoint of the
first rod and can also freely rotate around it, see Figure~\ref{fig:pend}. 

\begin{figure}[h]
\includegraphics{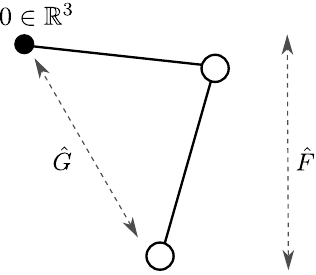}
\caption{The double pendulum defines two functions $\hF,\hG$ on $\CP^1\times \CP^1$.
}
\label{fig:pend}
\end{figure}

\noindent
Define two functions $\hF,\hG\co \CP^1\times\CP^1\to \R$ to be, respectively, the
$z$-coordinate of the free endpoint of the second rod, and its distance from the origin, normalised by $1/2$. In formulas,
$$
\begin{array}{l}
\CP^1 \times \CP^1=\{ x_1^2 + y_1^2 + z_1^2 = 1 \} \times \{ x_2^2 + y_2^2 + z_2^2 = 1 \} 
\subset \R^6,
\\
\hF(x_1,y_1,z_1,x_2,y_2,z_2) = \frac 1 2 (z_1 + z_2),
\\
\hG(x_1,y_1,z_1,x_2,y_2,z_2) = \frac 1 2 \sqrt{(x_1 + x_2)^2 + (y_1 + y_2)^2 +(z_1 + 
z_2)^2}.
\end{array}
$$
The function $\hG$ is not smooth along the anti-diagonal Lagrangian sphere
$\Sa = \{(x_1,y_1,z_1,x_2,y_2,z_2) \in \CP^1\times \CP^1; x_2 = -x_1, y_2 = -y_1, z_2= -z_1\}$
(corresponding to the folded pendulum), and away from it the 
functions $\hF$ and $\hG$ Poisson commute. The image of the ``moment map''
$(\hF,\hG)\co \CP^1\times\CP^1\to \R^2$ is the triangle shown in Figure~\ref{fig:polyt}.

\begin{figure}[h]
\includegraphics{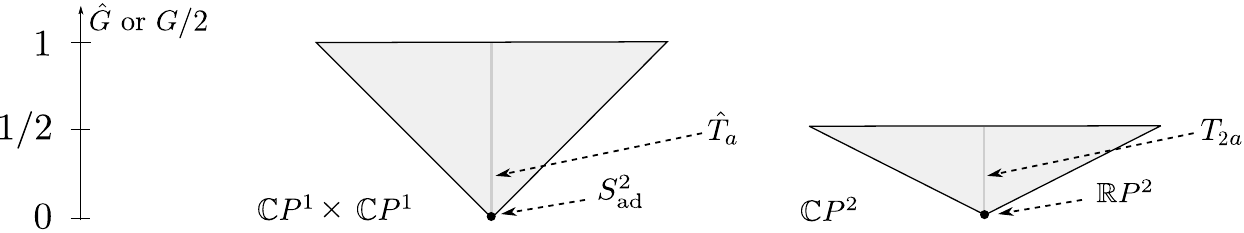}
\caption{The images of the ``moment maps'' on $\CP^1\times\CP^1$ and $\CP^2$, and the lines above which the tori $\T_a, T_a$ are located.
}
\label{fig:polyt}
\end{figure}

\begin{definition}
For $a\in(0,1)$, the Lagrangian torus $\T_a\subset \CP^1\times \CP^1$ is the pre-image of $(0,a)$
under the map $(\hF,\hG)$.
\end{definition}

The functions $(\hF,\hG)$ are invariant under the $\Z/2$-action on $\PxP$ that
swaps the two $\CP^1$ factors. This involution defines a 2:1 cover $\CP^1\times
\CP^1\to \CP^2$ branched along the diagonal of $\CP^1\times\CP^1$, so the
functions $(\hF,\hG)$ descend to functions on $\CP^2$ which we denote by $(F,G)$;
the image of $(F, G/2)\co \CP^2\to \R^2$ is shown in Figure~\ref{fig:polyt}.
Note that the quotient of the Lagrangian sphere $\Sa$ is $\RP^2 \subset \CP^2$.
Being branched, the 2:1 cover cannot be made symplectic, so it requires some
care to explain with respect to which symplectic form the tori $T_a\subset\CP^2$ are Lagrangian. One solution is to
consider $\CP^2$ as the symplectic cut \cite{Le95} of $T^*\RP^2$, as explained by
Wu~\cite{Wu15}. It is natural to take $(F,G/2)$, not $(F,G)$, as the ``moment
map'' on $\CP^2$.

We normalise the symplectic forms $\omega$ on $\CP^2$ and
 $\homega$ in $\PxP$ so that
 $\omega(\cH) = 1$ and $\homega(\cH_1) = \homega(\cH_2) = 1$, where $\cH =
 [\CP^1]$ is the generator of $H_2(\CP^2)$, and $\cH_1 = [\{\pt\}\times \CP^1]$,
 $\cH_2 = [\CP^1 \times \{\pt\}]$ in $H_2(\PxP)$. 
 
 
 

\begin{definition} 
For $a \in (0,1)$, the Lagrangian torus $T_a \subset \CP^2$ is the pre-image of
$(0, a/2)$ under $(F,G/2)$,~i.e.~the image of $\T_a$ under the 2:1 branched cover
$\PxP \to \CP^2$. 
\end{definition}

\begin{remark} \label{rem:def_Ta_Auroux} There is an alternative way to define
the tori $\T_a$ and $T_a$. It follows from the work of Gadbled \cite{Gad13}, see 
also \cite{OU13}, that the above defined tori are Hamiltonian isotopic to the
so-called Chekanov-type tori introduced by Auroux \cite{Au07}: 
$$
\begin{array}{cc} \T_a \cong \{ ([x:w],[y:z])\in \PxP \setminus
\{z=0\}\cup\{w=0\}: \ \frac{xy}{wz} \in \hgamma_a, \ \left|\frac{x}{w} \right| =
\left|\frac{y}{z} \right| \}, \\ T_a \cong \{ [x:y:z]\in \CP^2 \setminus
\{z=0\}: \ \frac{xy}{z^2} \in \gamma_a, \ \left|\frac{x}{z} \right| =
\left|\frac{y}{z} \right| \}, \end{array} 
$$ 
where $\hgamma_a,\gamma_a\subset
\C$ are closed curves that enclose a domain not containing $0\in\C$. The area of
this domain is determined by $a$ and must be such that the areas of holomorphic
disks computed in \cite{Au07} match Table~\ref{tab: Disks}; see below. (Curves
that enclose domains of the same area not containing $0\in\C$ give rise to
Hamiltonian isotopic tori.) The advantage of this presentation is that the tori
$T_a$ are immediately seen to be Lagrangian. The tori $\T_{1/2}$ and $T_{1/3}$ 
the monotone manifestations in $\PxP$ and $\CP^2$ of the Chekanov torus \cite{Che96}.
A presentation of the monotone Chekanov tori similar to the above was described 
in \cite{ElPo93}. 
Yet another way of defining the
tori is by Biran's circle bundle construction~\cite{Bi01} over a monotone circle
in the symplectic sphere which is the preimage of the top side of the triangles
in Figure~\ref{fig:polyt}; see again \cite{OU13}. \end{remark}

\subsection{Holomorphic disks} We start by recalling the theorem of Fukaya, Oh,
Ohta and~Ono mentioned in the introduction.

\begin{theorem}[{\cite[Theorem 3.3]{FO312}}]
For $a \in (0, 1/2]$, the torus $\T_a\subset \CP^1\times \CP^1$ is non-displaceable.\qed
\end{theorem}

\begin{proposition} \label{prp: Probes}
Inside $\CP^1\times\CP^1$ and $\CP^2$, 
all fibres corresponding to interior points of the ``moment polytopes'' shown in Figure~\ref{fig:polyt}, except for the tori $\T_a$ when $a\in(0,1/2]$, and $T_a$ when $a \in (0,1/3]$,
  are displaceable.  
\end{proposition}

\begin{proof} First, note that our model is toric in the complement of the
Lagrangians $\Sa \subset \PxP$ resp.~$\R P^2 \subset \CP^2$, represented
by the bottom vertex of Figure \ref{fig:polyt}. In fact, $\PxP \setminus \Sa$ and $\CP^2 \setminus \R P^2$ can be identified with the following normal
bundles, respectively: $\mathcal{O}(2)$ over the diagonal in $\PxP$, giving the maximum level set of
$\hG$; and $\mathcal{O}(4)$ over the conic in $\CP^2$, giving the
maximum level set of $G/2$. Clearly, these spaces are toric.

 Recall the method of probes due to McDuff~\cite{MD11} which is
a mechanism for displacing certain toric fibres. Horizontal probes displace all
the fibres except the $\T_a$ or $T_a$, $a\in (0,1)$. Vertical probes over the
segment $\{0\}\times (0,1/2]$ displace the $T_a$ for $a > 1/2$, and probes over
the segment $\{0\}\times (0,1]$ to displace the $\T_a$ for $a > 1/2$. All the
displacements given by probes can be performed by a Hamiltonian compactly
supported in the complement of the Lagrangians $S^2 \subset \PxP$, respectively
$\R P^2 \subset \CP^2$. When $1/3 < a< 1/2$, the method of probes cannot
 displace $T_a$. 

The proof of this remaining
case is due to Georgios Dimitroglou Rizell (currently not in the literature),
who pointed out that for $a > 1/3$, the tori $T_a$, up to Hamiltonian isotopy,
can be seen to project onto the open segment $S$ connecting $(0,0)$ with
$(1/3,1/3)$ in the standard moment polytope of $\CP^2$ (using the description of
Remark \ref{rem:def_Ta_Auroux}, we may take $\gamma_a$ inside the disk of radius
1 for $a > 1/3$). But there is a Hamiltonian isotopy of $\CP^2$ that sends the
preimage of $S$ to the preimage of the open segment connecting $(0,1)$ with
$(1/3,1/3)$, and hence disjoint from $S$. \end{proof}


The Maslov index 2 holomorphic disks for the tori $\T_a$ and $T_a$, with respect
to some choice of an almost complex structure for which the disks are regular,
were computed, respectively, by Fukaya, Oh, Ohta and~Ono \cite{FO312} and Wu
\cite{Wu15}. Their results can also be recovered using the alternative
presentation of the tori from Remark~\ref{rem:def_Ta_Auroux}. Namely, Chekanov
and Schlenk \cite{ChSch10} determined Maslov index 2 holomorphic disks for the
monotone Chekanov tori $T_{1/3} \subset \CP^2$ and $T_{1/2} \subset \PxP$, and
the combinatorics of these disks stays the same for the Chekanov-type tori from
Remark~\ref{rem:def_Ta_Auroux} if one uses the standard complex structures on
$\CP^2$ and $\PxP$ \cite[Proposition~5.8, Corollary~5.13]{Au07}.
We summarise these  results in the  statement below.

\begin{table}[h]
	{\it
	\begin{tabular}{|c|c|c|c|}
		\hline 
		\multicolumn{4}{|c|}{$\vphantom{\T_a}T_a\subset \CP^2$}
		\\
		\hline 
		Disk class&  \# & Area & $\PO$ term
		\\
		\hline 
		$\cH-2\beta-\alpha$ & 1 & $a$ & $t^az^{-2}w^{-1}$
		\\
		$\cH-2\beta$ & 2 & $a$ & $t^az^{-2}$
		\\
		$\cH-2\beta+\alpha$ & 1 & $a$ & $t^az^{-2}w$
		\\
		$\beta$ & 1 & $(1-a)/2$ & $t^{(1-a)/2}z$
		\\
		\hline		
		\multicolumn{4}{c}{}
	\end{tabular}
~
	\begin{tabular}{|c|c|c|c|}
		\hline
		\multicolumn{4}{|c|}{$\T_a\subset \CP^1\times\CP^1$}
		\\
		\hline
		Disk class&  \# & Area & $\PO$ term\\
		\hline 
		$\cH_1-\beta-\alpha$ & 1 & $a$ & $t^az^{-1}w^{-1}$\\
		$\cH_1-\beta$ & 1 & $a$ & $t^az^{-1}$ \\
		$\cH_2-\beta$ & 1 & $a$ & $t^az^{-1}$\\
		$\cH_2-\beta+\alpha$ & 1 & $a$ & $t^az^{-1}w$\\
		$\beta$ & 1 & $1-a$ & $t^{1-a}z$\\
		\hline
	\end{tabular}
	}
	\smallskip
	\caption{The homology classes of all Maslov index two $J$-holomorphic disks on the tori; the
		number of such disks through a generic point on the torus; their areas; the corresponding monomials in the superpotential function: all for
		some regular almost complex structure $J$. Here
		$\alpha,\beta$ denote some fixed homology classes in $H_2(\CP^2,T_a)$ or 
		$H_2(\CP^1\times\CP^1,\T_a)$, and $\bd \alpha, \bd \beta$ generate $H_1(T_a, \Z)$
		or $H_1(\T_a, \Z)$.}
\label{tab: Disks}
\end{table}

\begin{proposition}[\cite{Au07, ChSch10, FO312, Wu15}] \label{Prop: Disks T_a} 
There exist almost complex structures on $\CP^2$ and $\PxP$ for which the
enumerative geometry of Maslov index 2  holomorphic disks with boundary on $T_a$,
resp.~$\T_a$, is as shown in Table~\ref{tab: Disks}, and these disks are regular.
Here we are considering the standard spin structure in the tori to orient
the moduli spaces of disks. \qed
\end{proposition}

\begin{remark}
 The fact that all disks contribute with positive signs is an argument 
 analogous to \cite[Proposition~8.2]{Cho04} --  
 see also \cite[Section~5.5]{Vi13} for a similar discussion. 
\end{remark}

\subsection{Proof of Theorem \ref{th:T_a}}
We now have all the ingredients to prove Theorem~\ref{th:T_a} using
Theorem~\ref{th:CO}. 
Take the almost complex structure $J$ from Proposition~\ref{Prop: Disks T_a}, then the parameter $a$ indexing the torus $T_a\subset\CP^2$ satisfies Equation~\eqref{eq:defn_a_dim4} whenever $a<1/3$. 
Let $\{D_i\}_i\subset (\CP^2,T_a)$ be the images of all
$J$-holomorphic Maslov index 2 disks of area $a$ such that $p \in \bd D_i$, for
a fixed point $p \in T_a$. 
We work over the coefficient ring $Q=\Z /8$. According to Table~\ref{tab: Disks},
$$
\sum_i \bd [D_i]= - 8\cdot \bd\beta = 0 \in H_1(T_a; \Z/8).
$$ 
Moreover, according to Table~\ref{tab: Disks} we have

\begin{equation} \label{eq: CP2OC_Ta}
 \OC\clow([p_{T_a}]) = 4\cH \in H_2(\CP^2; \Z/8). 
\end{equation} 
Note that the next to the least area $A$ from Equation~(\ref{eq: def A}) equals $A = (1-a)/2$. 

 Let us move to the Clifford torus. It is well known that
the monotone Clifford torus $T_{\Cl}$ bounds three Maslov index 2
$J$-holomorphic disks passing through a generic point, belonging to classes of
the form $\beta_1$, $\beta_2$, $\cH - \beta_1 - \beta_2\in H_2(\CP^2,T_\Cl;\Z)$
(and counting positively with respect to the standard spin structure on the
torus) \cite{CO06}, see also \cite[Proposition~5.5]{Au07}, and having area $b =
1/3$. So we obtain
$$ 
\OC\c([p_{T_\Cl}]) = \cH \in H_2(\CP^2; \Z/8).
$$
\begin{proof}[Proof of Theorem \ref{th:T_a}]
	Since 
	$$\OC\clow([p_{T_a}]) \cdot \OC\c([p_{T_\Cl}]) = 4 \ne 0 \mod 8,$$ we are in shape to
	apply Theorem \ref{th:CO}, provided that:
	$$
	a + b = a + 1/3 < A = \tfrac{1-a}{2}
	$$
	~i.e.~$a < 1/9$. The case $a = 1/9$ follows by continuity. 
\end{proof}

\begin{remark}
\label{rem:not_works_for_conj}
We are unable to prove that the tori $T_a$ are non-displaceable  using Theorem~\ref{th:CO} because
$
\OC\clow([p_{T_a}])\cdot \OC\clow([p_{T_a}]) = 16 \equiv 0\mod 8.
$ 
\end{remark}

\begin{remark}
\label{rem:not_works_over_Z}
It is instructive to see why the argument cannot be made to work over $\C$ or $\Z$. Then $\sum_i \bd [D_i]= - 8\cdot \bd\beta$ is non-zero, but this can be fixed by introducing a local system $\rho\co \pi_1(T_a)\to\C^\x$ taking $\alpha\mapsto -1$, $\beta\mapsto +1$. By definition, $\rho$ is multiplicative, so for example, $\rho(\alpha+\beta)=\rho(\alpha)\rho(\beta)$. Then $\sum_i\rho(\bd[D_i])\cdot \bd[D_i]$ equals
$$
-(-2\bd\beta-\bd\alpha)+2(-2\bd\beta)-(-2\bd\beta+\bd\alpha)=0\in H_1(T_a;\C).
$$
However, in this case $\OC\clow([p_{T_a}];\rho)=\sum_i\rho(\bd [D_i])[D_i]$
vanishes in $H_2(\CP^2;\C)$, because the $H$-classes from Table~\ref{tab: Disks}
cancel in this sum. \end{remark}

\subsection{Similar theorems for $\PxP$ and $\BlIII$} 
Using our technique, we
can prove a similar non-displaceability result inside $\PxP$, which is probably
less novel, and $\BlIII$, both endowed with a monotone symplectic form. We start 
with $\PxP$.

\begin{theorem}
\label{th: Ta in PxP}
	For each $a\in (0,1/4]$, the torus $\T_a \subset \PxP$ is Hamiltonian
	non-displaceable from the monotone Clifford torus $T_\Cl \subset \PxP$.  
\end{theorem}

\begin{remark}  We believe this theorem can be obtained by a short elaboration
on~\cite{FO312}: for the bulk-deformation $\bb$ used in \cite{FO312}, there
should exist local systems (which in this context are weak bounding cochains
\cite[Appendix~1]{FO312}) on $\T_a$ and $T_\Cl$ such that $HF^{\bb}(\T_a,T_\Cl)
\ne 0$, for $a \in (0, 1/2]$. Alternatively, in addition to
$HF^{\bb}(\T_a,\T_a)\neq 0$ as proved in \cite{FO312}, one can show that
$HF^{\bb}(T_\Cl,T_\Cl)\neq 0$ for some local system, and there should be a
bulk-deformed version of Theorem~\ref{th:OCCO_monot} using the unitality of the
string maps and the semi-simplicity of the deformed quantum cohomology
$QH^{\bb}(\CP^2)$. Our proof only works for $a\le 1/4$, but is based on much
simpler transversality foundations. \end{remark}

As a warm-up, let us try to apply Theorem~\ref{th:CO}; we shall work over $\Z/4$. By looking at Table~\ref{tab: 
Disks}, we see that for $a < 1/2$ we have
	 \begin{equation} \label{eq: PxP:OC_Ta}
	   \OC\clow([p_{\T_a}]) = 2(\cH_1 + \cH_2) \in H_2(\PxP; \Z/4),
	 \end{equation}
	and $ A = 1 -a$.
	One easily shows that 
	$$\OC\c([p_{\T_\Cl}]) = \cH_1 + \cH_2 \in H_2(\PxP; \Z/4),$$ 
 since the Clifford torus bounds holomorphic Maslov index 2 disks of area $b =
 1/2$, passing once through each point of $\T_\Cl$, in classes of the form
 $\beta_1$, $\beta_2$, $\cH_1 - \beta_1$, $\cH_2 - \beta_2$ (and counting
 positively with respect to the standard spin structure on the torus)
  \cite{CO06}, see also
 \cite[Section~5.4]{Au07}. We cannot directly apply Theorem~\ref{th:CO} because
	$$\OC\clow([p_{\T_a}]) \cdot \OC\c([p_{\T_\Cl}]) = 4 \equiv 0 \mod 4.$$ 
Hence we need to use the more refined Theorem \ref{th:CO_groups}. 	
\begin{proof}[Proof of Theorem \ref{th: Ta in PxP}] 
 
Consider $S_{\T_\Cl} \subset H_1(\T_\Cl; \Z/2)$ to be the linear space generated
 by $[\del \beta_2]$ and $S_{\T_a} \subset H_1(\T_a; \Z/2)$ generated by $\del
 \beta$; both satisfy Condition~ \eqref{eq:disk_groups_cancel_bdies} over $\k=Q=\Z/2$. So we have:
 
 \begin{equation} \label{eq: PxP:OC_Ta_S}
   \OC\clow([p_{\T_a}],S_{\T_a}) = \cH_1 + \cH_2 \in H_2(\PxP; \Z/2),
 \end{equation}
	
\begin{equation}
  \OC\c([p_{\T_\Cl}],S_{\T_\Cl}) = \cH_2 \in H_2(\PxP; \Z/2),
\end{equation}	
and hence,

$$\OC\clow([p_{\T_a}],S_{\T_a}) \cdot \OC\c([p_{\T_\Cl}],S_{\T_\Cl}) = 1 \neq 0 \mod 2.$$ 
Therefore by Theorem \ref{th:CO_groups}, $\T_a$ is non-displaceable from $\T_\Cl$
provided that $ a + b = a + 1/2 < A = 1 - a$, i.e.~ $a < 1/4$.
\end{proof}

	

Next, we pass on to $\BlIII$ which we see as $\PxP$ blown up at the two points
corresponding to the two top corners of the image of the ``moment map''
$(\hF,\hG)$, see Figure \ref{fig:polytBlIII}. If the blowup is of the correct
size then the resulting symplectic form on $\BlIII$ is monotone; see
\cite[Section 7]{Vi16b} for more details. We denote by $\bar{T}_a$ the tori in
$\BlIII$ coming from the $\T_a \subset \PxP$, in particular, $\bar{T}_a =
L^{1/2}_{1 - a}$ in the notation of \cite[Section 7]{Vi16b}. We also denote by
$\bar{T}_{\Cl}$ the monotone torus corresponding to the baricentre of the
standard moment polytope of $\BlIII$.

\begin{figure}[h]
\includegraphics{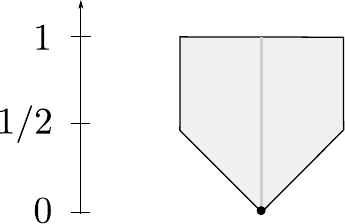}
\caption{The images of the ``moment maps'' on $\BlIII$, 
and the line above which the tori $\bar{T}_a$ are located.
}
\label{fig:polytBlIII}
\end{figure}

\begin{theorem}
\label{th: Ta in BlIII}
	For each $a\in (0,1/4]$, the torus $\bar{T}_a \subset \BlIII$ is Hamiltonian
	non-displaceable from the monotone Clifford torus $\bar{T}_\Cl \subset \BlIII$.
\end{theorem}

\begin{proof} 
Let $E_1$ and $E_2$ be the classes of the exceptional curves of the above blowups, so that 
$$H_2(\BlIII,\bar{T}_a) = \langle \cH_1,
\cH_2, E_1, E_2, \beta, \alpha \rangle.$$ Compared to Table~\ref{tab: Disks}, the torus $\bar{T}_a$
acquires two extra holomorphic disks of area $1/2$, with boundary in classes
$[\del \alpha]$ and $-[\del \alpha]$, and whose sum gives the class $\cH_1
+ \cH_2 - E_1 - E_2$, see \cite[Lemma 7.1]{Vi16b}.

We then use $S_{\bar{T}_a} \subset H_1(\bar{T}_a; \Z/2)$ generated by $\del
\beta$ and $S_{\bar{T}_\Cl} \subset H_1(\bar{T}_\Cl; \Z/2)$ in a similar
fashion as in the proof of Theorem \ref{th: Ta in PxP}, so that $S_{\bar{T}_a}$,
$S_{\bar{T}_\Cl}$ satisfy Condition~\eqref{eq:disk_groups_cancel_bdies} and
$\OC\c([p_{\T_\Cl}],S_{\T_\Cl}) = \cH_2$. Hence
$$\OC\clow([p_{\bar{T}_a}],S_{\bar{T}_a}) \cdot \OC\c([p_{\bar{T}_\Cl}],S_{\bar{T}_\Cl})
 = (H_1+H_2)\cdot H_2 =1 \mod 2.$$
If one defines $A$ by (\ref{eq: def A}), then $A = b = 1/2$, so
Theorem~\ref{th:CO_groups} does not apply. However, we can use Theorem~\ref{thm:alg_cancel}.
Notice that the boundaries of both disks of area $1/2$ are equal to $\alpha$ over $\Z/2$, and there are two such disks so their count vanishes over $\Z/2$. Therefore in the setup of Theorem~\ref{thm:alg_cancel} we can take $A=
1-a$. So we get the desired non-displaceability result as long as $a+b < 1-a$, i.e.~$a
< 1/4$.
\end{proof}

\begin{remark}
We expect that Theorems~\ref{th: Ta in PxP},~\ref{th: Ta in BlIII} can be improved
so as to allow $a\in (0,1/2]$. Indeed, the tori appearing in those theorems are non-self-displaceable for $a\in(0,1/2]$: see \cite{FO312, FO311b} for the case of $\CP^1\times \CP^1$ and \cite{Vi16b} for the case of
$\BlIII$; and see the previous remark.  
\end{remark}

\subsection{Proof of Lemma~\ref{lem:cotangent_bundles}} Starting with
$X=\CP^1\times \CP^1$ or $X=\CP^2$, remove the divisor $D\subset X$ given by the
preimage of the top side of the triangle in Figure~\ref{fig:polyt} under the
``moment map''. The complement $U$ is symplectomorphic to an open co-disk bundle
inside $T^*S^2$, respectively $T^*\RP^2$. The Lagrangian tori $\hat T_a$ resp.~$
T_a$ are monotone in $U$. Indeed, note that the only disk in $X$ passing through
the divisor $D$ is the one in class $\beta$ (Table \ref{tab: Disks}) -- this can
be seen in any presentation of the tori \cite{Au07,FO312,Wu15}, see again Remark
\ref{rem:def_Ta_Auroux}. 
Monotonicity of the tori follows from noting that
$H_2(U,\T_a; \Q)$, resp. $H_2(U,T_a; \Q)$, is generated by the remaining Maslov
index 2 disks -- which all have the same area $a$ and boundary generate
$H_1(\T_a; \Q)$, resp. $H_1(T_a; \Q)$ -- together with the Lagrangian
zero-section $\Sa$ when $X = \PxP$, which have Maslov index $0$ (recall that
$H_2(T^*\R P^2; \Q) = 0$). Actually these tori differ by scaling inside the
respective cotangent bundle. We denote these tori seen as sitting in the
cotangent bundles by $\hat L_a \subset T^*S^2$ resp.~$L_a\subset T^*\RP^2$.
These are the tori we take for Lemma~\ref{lem:cotangent_bundles}. In the
cotangent bundle, the tori can be scaled without constraint so we actually get a
family indexed by $a\in(0,+\infty)$ and not just $(0,1)$.

As we pointed out, the holomorphic disks of area $a$ from Table \ref{tab: Disks}
are precisely the ones which lie in the complement of $D\subset X$ \cite{FO312,
Wu15}, therefore they belong to $U$. Finally, the tori $\hat L_a$ and $L_a$
bound no holomorphic disks in $T^*S^2$ resp.~$T^*\RP^2$ other than the ones
contained inside $U$, by the maximum principle. Therefore we know all
holomorphic Maslov index 2 disks on these tori, and
Lemma~\ref{lem:cotangent_bundles} becomes a straightforward computation,
that we actually already performed. Indeed, the disks used to compute
$\OC\c$ in $U$ (and hence in $T^*S^2$ resp.~$T^*\RP^2$) are
the same used to compute $\OC\clow$ in $X$, i.e., 
$\OC\clow([p_{\T_a}]) =i_*\OC\c([p_{\hat L_a}])$, resp. 
$\OC\clow([p_{T_a}]) =i_*\OC\c([p_{L_a}])$.

\begin{remark}
  Note that the disks computed in Table \ref{tab: Disks} were with respect to 
  the standard complex structure $J$. Moreover, the divisor $D$ corresponds to 
  the diagonal in $\PxP$ and to a conic in $\CP^2$. 
   In particular, $J$ is cylindrical at infinity
  for $X\setminus D$.
 \end{remark}

Namely, as in the proof of Theorem \ref{th: Ta in PxP}, the holomorphic Maslov index 2 disks
with boundary on $\hat L_a \subset T^*S^2$ satisfy Condition
\eqref{eq:disk_low_area_cancel_bdies} over $\Z/4$, and Equation~\eqref{eq:S2_OC} from Lemma
\ref{lem:cotangent_bundles} follows immediately from \eqref{eq: PxP:OC_Ta}: 
$$
\begin{array}{r}
i_*\OC\c([p_{\hat L_a}]) = \OC\clow([p_{\T_a}]) = 2(\cH_1 + \cH_2) = 
2(\cH_1 - \cH_2)
\\ = 2i_*[S^2] \in
H_2(\PxP; \Z/4),
\end{array}
$$ 
and injectivity of $i_*: H_2(U; \Z/4) \to H_2(\PxP; \Z/4)$, where $i$ is the embedding of $U\subset X$..

Similarly, we can identify $S_{\hat L_a}$ with the $S_{\T_a}$ from proof of
Theorem \ref{th: Ta in PxP}, which satisfies Condition
\eqref{eq:disk_groups_cancel_bdies} over $\k=Q=\Z/2$. Equation~\eqref{eq:S2_OC_groups} from
Lemma \ref{lem:cotangent_bundles} follows immediately from \eqref{eq:
PxP:OC_Ta_S}:
$$
\begin{array}{r}
i_*\OC\c([p_{\hat L_a}],S_{\hat L_a}) = \OC\clow([p_{\T_a}],S_{\T_a}) = \cH_1 + \cH_2 = \cH_1 - \cH_2 
\\
=  i_*[S^2] 
\in H_2(\PxP; \Z/2),
\end{array}
$$
and the injectivity of $i_*: H_2(U; \Z/2) \to H_2(\PxP; \Z/2)$. 

Analogously, Lemma \ref{lem:cotangent_bundles}(ii)
is checked as in the proof of Theorem \ref{th:T_a}, in particular Equation~\eqref{eq:RP2_OC}
 follows from \eqref{eq: CP2OC_Ta}:
$$
i_*\OC\c([p_{L_a}]) = \OC\clow([p_{T_a}]) = 4\cH = i_*[4\R P^2] \in H_2(\CP^2; \Z/8).
$$ 
Indeed,
$i_*$ sends the generator $[4\R P^2]$ of $H_2(T^*\R P^2; \Z / 8) \cong \Z / 2$
to $4\cH\allowbreak \in H_2(\CP^2;$ \linebreak[3] $\Z/8)$.

Finally, we note that these computations are actually valid for $a \in
(0,+\infty)$, as scaling monotone tori in a cotangent bundle does not change the
enumerative geometry of holomorphic disks. \qed


\subsection{The superpotentials} We conclude by an informal discussion of
the superpotentials of the tori we study, aimed to readers familiar to the
notions of the superpotential and bulk deformations. We refer to
\cite{Au07,FO310b,FO312,Wu15} for the definitions. The Landau-Ginzburg
superpotential (further called ``potential'') associated to a Lagrangian 2-torus
and an almost complex structure $J$ is a Laurent series in two variables which
combinatorially encodes the information about all $J$-holomorphic index 2 disks
through a point on $L$.
In the setting of Proposition~\ref{Prop: Disks T_a}, the potentials are given by
\begin{equation}\label{eq:Pot_CP2}
\PO_{\CP^2}  = t^{(1-a)/2}z + \frac{t^a}{z^2w} + 2\frac{t^a}{z^2} + \frac{t^aw}{z^2} = t^{(1-a)/2}z + t^a\frac{(1 +w)^2}{z^2w};
\end{equation}
\begin{equation}\label{eq:Pot_CP1}
\PO_{\PxP}  = t^{1-a}z + \frac{t^a}{zw} + 2\frac{t^a}{z} + \frac{t^aw}{z} = t^{1-a}z + t^a\frac{(1 + 
	w)}{zw} + t^a\frac{(1 +  w)}{z}.
\end{equation}

(These functions are sums of monomials corresponding to the disks as shown in Table~\ref{tab: Disks}.)
Here $t$ is the
formal parameter of the Novikov ring $\Lambda_0$ associated with a ground field $\k$, usually assumed to be of characteristic zero:
$$
\Lambda_0 = \{ \sum a_i t^{\lambda_i}\ | \ a_i \in \k, \ \lambda_i \in \R_{\ge 0}, \
\lambda_i \le \lambda_{i+1}, \ \lim_{i \to \infty} \lambda_i = \infty \}. 
$$

Let
$\Lambda_\x$ be the field of elements of $\Lambda_0$ with nonzero constant term $a_0t^0$.
We can see $(\Lambda_\x)^2$  as the space of local systems $\pi_1(L)\to \Lambda_\x$ on a Lagrangian torus $L$, or \cite[Remark 5.1]{FO310b} as the space $\exp(H_1(L;\Lambda_0))$ of exponentials of elements in $H_1(L;\Lambda_0)$, the so-called bounding cochains from the works of Fukaya, Oh, Ohta and~Ono \cite{FO310,FO310b, FO3Book}.
In turn, the potential can be seen as a function $(\Lambda_\x)^2\to \Lambda_0$, 
and its critical points
correspond to local systems
$\sigma\in (\Lambda_\x)^2$ such that $HF^*(L,\sigma)\neq 0$ \cite[Theorem 5.9]{FO310b}


If the potential has no critical points, it  can sometimes be fixed by  introducing a bulk deformation $\bb \in H^{2k}(X; \Lambda_0)$ which deforms the function; critical points of the deformed potential correspond to local systems $\sigma\in (\Lambda_\x)^2$ such that $HF^\bb(L,\sigma)\neq 0$ \cite[Theorem~8.4]{FO310b}. This was the strategy of \cite{FO312} for proving that the tori $\T_a\subset\CP^1\times\CP^1$ are non-displaceable. When $\bb\in H^{2}(X;\Lambda_0)$, the deformed potential is still determined by Maslov index 2 disks (if $\dim X=2n>4$, this will be the case for $\bb\in H^{2n-2}(X;\Lambda_0)$), see e.g.~\cite[Theorem~8.2]{FO310b}. For bulk deformation classes in other degrees, the deformed potential will use disks of all Maslov indices, and its computation becomes out of reach.

In contrast to the $\T_a$, the potential for the tori $T_a$ does not acquire a critical point after we introduce a degree 2 bulk deformation class
$\bb \in H^2(\CP^2, \Lambda_0)$.

\begin{proposition} \label{prop: Bulk CP^2}
Unless $a=1/3$, for any bulk deformation class $\bb \in H^2(\CP^2, \Lambda_0)$, the deformed potential $\PO^\bb$ for the torus $T_a\subset \CP^2$ has no critical point in $(\Lambda_\x)^2$.
\end{proposition}

\begin{proof}
Let $Q\subset\CP^2$ be the quadric which is the preimage of the top side of the traingle in Figure~\ref{fig:polyt}, so $[Q]=2H$. 
Then $\bb$ must be Poincar\'e dual to $c \cdot [Q]$ for some $c \in \Lambda_0$. Among the holomorphic disks in Table~\ref{tab: Disks}, the only disk intersecting $Q$ is the $\beta$-disk intersecting it once \cite{Wu15}. Therefore the deformed potential
$$
   \PO^{\bb}_{\CP^2} = t^{(1-a)/2}e^c z + t^a\frac{(1 + w)^2}{z^2w} 
  $$
differs from the usual one by the $e^c$ factor by the monomial corresponding to the $\beta$-disk, compare \cite{FO312}. Its critical points are given by
   $$
   w = 1 , \ z^3 = 8t^{(3a-1)/2}e^{-c}. 
   $$
Unless $3a-1=0$, the $t^0$-term of $z$ has to vanish, so $z\notin\Lambda_\x$.
 \end{proof}


Keeping an informal attitude,
let us drop the monomial $t^{(1-a)/2}z$ from Equation~(\ref{eq:Pot_CP2}) of $\PO_{\CP^2}$; denote the resulting function by $\PO_{\CP^2,\low}$. 
For $a < 1/3$, it reflects the information about the least area holomorphic disks with boundary on~$T_a\subset\CP^2$,
\begin{equation}
\label{eq:trunc_po}
\PO_{\CP^2,low} =  t^a\frac{(1 + w)^2}{z^2w}. 
\end{equation}
Now, this function has plenty of critical points. Over $\C$, it has the critical line $w=-1$, and if one works over $\Z/8$ then the point $(1,1)$ is also a critical point, reflecting the fact the boundaries of the least area holomorphic Maslov index 2 disks on $T_a$ cancel modulo 8, with the trivial local system.

The potential (\refeq{eq:trunc_po}) becomes the usual potential for the monotone
tori $L_a\subset T^*S^2$ from Lemma~\ref{lem:cotangent_bundles}. The fact that
it has a critical point implies, this time by the standard machinery, that the
tori $L_a\subset T^*\RP^2$ are non-displaceable \cite[Theorem~2.3]{FO312}  
(note Condition~6.1, Theorem~A.1 and Theorem~A.2 in \cite[Appendix~1]{FO312}), see also \cite{BC12,She13}.
The same is true of the $\hat L_a\subset T^*S^2$ and has been known due to
\cite{AF08}, see also \cite[Appendix~2]{FO312}.


\bibliography{Symp_bib}{}
\bibliographystyle{plain}

\end{document}